\documentclass[a4paper]{article}
\usepackage{latexsym,bm}
\usepackage{amssymb,amsmath,amsthm}
\usepackage[english,german]{babel}
\usepackage[colorlinks=true]{hyperref}
\usepackage{color}
\usepackage{times}
\allowdisplaybreaks[4]

\newtheorem{theorem}{Theorem}[section]
\newtheorem{lemma}[theorem]{Lemma}
\newtheorem{remark}[theorem]{Remark}
\newtheorem{proposition}[theorem]{Proposition}
\newtheorem{definition}[theorem]{Definition}
\newtheorem{corollary}[theorem]{Corollary}
\numberwithin{equation}{section}
\hypersetup{linkcolor=blue,urlcolor=red,citecolor=red}
\title{Rapid stabilizability of infinite-dimensional control systems with time delays{\footnote{This work was supported by the National Natural Science Foundation of China
under grant 12171377, 12571483.}}}
\author{Yaxing Ma{\footnote{School of
Mathematics and Statistics, Wuhan University, Wuhan 430072, China;
e-mail: yaxingma@yeah.net.}}\and Lijuan Wang{\footnote{School of
Mathematics and Statistics, Wuhan University, Wuhan 430072, China;
e-mail: ljwang.math@whu.edu.cn.}}\and Huaiqiang Yu{\footnote{School of Mathematics and KL-AAGDM, Tianjin University, Tianjin, 300350, China;
email: huaiqiangyu@yeah.net.}}}
\date{}
\begin{document}
\selectlanguage{english}
\maketitle
\begin{abstract}
    In this paper, we investigate the rapid stabilizability of linear infinite-dimensional control systems with a constant time delay. Under the assumptions that the state operator generates an immediately compact semigroup and that the delay coefficient is constant, we establish two main results:
(i) The presence of a time-delay term does not affect the rapid stabilizability of the control system; that is, this property depends only on the state and control operators;
(ii) Static feedback is sufficient to achieve rapid stabilization of the system.
Applications are also presented.
\end{abstract}


\section{Introduction}\label{yu-section-1}
 Studies on infinite-dimensional delayed (or retarded) control systems date back to the 1970s (see, for instance, \cite{Delfour-Macalla-Mitter, Datko-Lagnese-Polois, Yamamoto} and the references therein). Since then, significant progress has been made in this area. Among the various approaches, the state-space approach (see, e.g., \cite{Nakagiri}) has proved to be particularly influential. This approach reformulates delayed systems as coupled systems, which can in turn be interpreted as abstract evolution equations within the $C_0$-semigroup framework (see Section \ref{sec-yu-10-1-2}), thereby greatly facilitating their analysis.
In particular, when combined with duality arguments in control theory, the state-space approach allows one to characterize the controllability properties of delayed systems in terms of observability properties of the corresponding adjoint systems (see, e.g., \cite{Khodja-Bouzidi-Dupaix-Maniar}). In this paper, we focus on the stabilizability of infinite-dimensional delayed control systems.

We now provide a brief overview of several related results in this area.
In \cite{Jeong}, the  stabilizability of a general delayed control system in Hilbert space was investigated, and a spectral characterization of stabilizability was obtained under the assumptions that the state operator generates an analytic semigroup, the associated coupled operator has only point spectrum, and the control operator is bounded. The approach used in \cite{Jeong} builds upon \cite{Prato-Lunardi}, where the $L^2$-stabilizability of integro-differential parabolic equations was studied. It is worth noting that $L^2$-stabilizability
in the setting of \cite{Jeong} is equivalent to exponential stabilizability (see \cite[Theorem 6.2.7]{Curtain-Zwart}).
Further results for systems with constant delay were obtained in \cite{Nakagiri-Yamamoto}, where stabilizability was established under the assumptions that the state operator generates an immediately compact semigroup, the control operator is bounded, and the control space is finite-dimensional. These results were later extended to more general abstract settings in \cite{Hadd-Zhong}. More recently, \cite[Section 4.4]{Kunisch-Wang-Yu} addressed the stabilizability of a delayed heat equation on the whole space, where the state operator has only continuous spectrum and does not generate a compact semigroup.
In parallel, considerable attention has been devoted to the stabilization of delay-free systems via delayed feedback. For instance, the backstepping method and the Artstein transformation were employed in \cite{Krstic} and \cite{Prieur-Trelat}, respectively, to design delayed boundary feedback laws for reaction-diffusion equations.

From the above literature, it is evident that most existing works on delayed control systems focus on exponential stabilizability, while rapid stabilizability has received comparatively little attention. One possible explanation is the prevailing view that, for open-loop systems, rapid stabilizability is, to a large extent, equivalent to null controllability. However, recent studies \cite{Kunisch-Wang-Yu, Liu-Wang-Xu-Yu} provide examples of systems that are rapidly stabilizable but not null controllable. Meanwhile, it has been shown in \cite{Khodja-Bouzidi-Dupaix-Maniar, Kunisch-Wang-Yu} that the presence of delay may even destroy the null controllability of certain parabolic systems.
These observations reveal a sharp contrast: while delays may severely deteriorate controllability properties, their impact on rapid stabilizability remains largely unclear. This leads to the following fundamental issue:

\begin{enumerate}
\item[(P1)] \emph{Does the presence of delay affect rapid stabilizability?}
\end{enumerate}
    Moreover, it is well known that feedback stabilization involves both the design of feedback laws, also referred to as feedback operators, and the observation of system states. For delayed systems, the evolution depends on past states, which gives rise to two types of feedback: dynamic feedback, involving both the current and past states, and static feedback, which depends solely on the current state.
   As mentioned above, the state-space approach is the classical method for treating control systems with time delays. However, feedback laws derived via this approach are typically of a dynamic type. Therefore, simplifying the structure of the feedback law is a meaningful objective. This raises the following issue:
\begin{enumerate}
\item[(P2)] \emph{Can rapid stabilization be achieved using only static feedback?}
\end{enumerate}

The main purpose of this paper is to address Problems (P1) and (P2). However, tackling these problems is far from straightforward.
Indeed, the delay term can be interpreted as a short-term forcing acting on the system. For linear systems subject to such forcing, exponential stabilizability may fail when the range of the control operator does not coincide with the entire state space. From this perspective, delayed control systems can be viewed as control systems with forcing terms.
Fortunately, by means of the state-space approach, such systems can be transformed into coupled systems driven solely by the control (see \cite{Kunisch-Wang-Yu, Nakagiri} or Section \ref{sec-yu-10-1-2}), thereby enabling the application of techniques developed for abstract linear control systems.
\vskip 5pt
\noindent
\textbf{Notations.}
Let $\mathbb{N}^+ := \mathbb{N}\setminus\{0\}$, $\mathbb{R}^+ := (0,+\infty)$, $\mathbb{C}_{\gamma}^+ := \{z \in \mathbb{C} : \mathrm{Re}z > \gamma\}$, and $\mathbb{C}_{\gamma}^- := \{z \in \mathbb{C} : \mathrm{Re} z < \gamma\}$ for $\gamma \in \mathbb{R}$.
Given a Banach space $X$, we denote its norm and dual space by $\|\cdot\|_X$ and $X'$, respectively. When $X$ is a Hilbert space, its inner product is denoted by $\langle \cdot, \cdot \rangle_X$.
Given two Banach spaces $X_1$ and $X_2$, we denote by $\mathcal{L}(X_1; X_2)$ the space of all bounded linear operators from $X_1$ to $X_2$, equipped with the usual operator norm. If $X_1 = X_2$, we simply write $\mathcal{L}(X_1) := \mathcal{L}(X_1; X_1)$.
For a (possibly unbounded) linear operator $L$ from $X_1$ to $X_2$, we denote its domain by $D(L) := \{ f \in X_1 : Lf \in X_2 \}$, its adjoint by $L^*$, and its resolvent set and spectrum by $\rho(L)$ and $\sigma(L)$, respectively.
If $\Lambda$ is a closed subspace of $X$, its dimension is denoted by $\mbox{Dim}(\Lambda)$. Finally, $C(\cdots)$ denotes a positive constant whose value may depend on the parameters specified in the parentheses.
\subsection{Control system}
     Let  $X$ and $U$ be two complex separable Hilbert spaces, with $X'=X$ and $U'=U$. Let $\tau\in\mathbb{R}^+$ and  $\kappa\in\mathbb{R}$.
We consider the following control system:
\begin{equation}\label{yu-11-23-1}
\begin{cases}
    y_t(t)=Ay(t)+\kappa y(t-\tau)+Bu(t),&t\in\mathbb{R}^+,\\
    y(0)=y_0,\\
    y(t)=\varphi(t),&t\in(-\tau,0),
\end{cases}
\end{equation}
    where  $u\in L^2_{loc}(\mathbb{R}^+;U)$, $y_0\in X$, and $\varphi\in L^2(-\tau,0;X)$. The state operator $A$ and the control operator $B$
    satisfy the following assumptions:
\begin{enumerate}
     \item[$(A_1)$] The operator $A:D(A)(:=X_1)\subset X\to X$ generates  an immediately  compact semigroup $S(\cdot)$ on $X$.
\item[$(A_2)$] The operator $B$ belongs to $\mathcal{L}(U;X_{-1})$, where $X_{-1}$ is the completion of $X$ with respect to the norm $\|f\|_{-1}:=\|(\rho_0I-A)^{-1}f\|_{X}$ ($f\in X$). Here and throughout the paper,
    $\rho_0\in \rho(A)\cap \mathbb{R}$ is fixed.

    \item[$(A_3)$] There exist  $T>0$ and  $C(T)>0$ such that
\begin{equation}\label{yu-5-4-2}
    \int_0^T\|B^*S^*(t)\varphi\|_{U}^2\mathrm dt \leq C(T)\|\varphi\|_{X}^2\;\;\mbox{for any}\;\;\varphi\in D(A^*).
\end{equation}
\end{enumerate}

\begin{remark}\label{yu-remark-8-00-1}
We make several remarks on system \eqref{yu-11-23-1} and Assumptions $(A_1)$-$(A_3)$.
\begin{enumerate}
\item [$(i)$] Assumption $(A_1)$ implies that both $S(t)$ and $S^*(t)$ are compact for all $t>0$. Hence, the almost exponential decay condition (AEDC) introduced in \cite{Ma-Trelat-Wang-Yu} holds (see \cite[Remark 1.8]{Ma-Trelat-Wang-Yu}).
\item [$(ii)$]  For simplicity, we denote by $X^*_1$ the space $D(A^*)$ endowed with the norm $\|(\rho_0I-A^*)f\|_X$  for $f\in D(A^*)$. It is well known that $X_{-1}$ is the dual space of $X^*_1$ with respect to the pivot space $X$ (see \cite[Section 2.9, Chapter 2]{Tucsnak-Weiss}), and we denote the corresponding duality pairing by $\langle\cdot,\cdot\rangle_{X_{-1},X^*_1}$.
\item [$(iii)$] The operator $A$ with domain $X_1$ admits a unique extension $\widetilde{A}\in \mathcal{L}(X;X_{-1})$ in  the following sense:
\begin{equation*}\label{yu-5-4-3}
    \langle \widetilde{A}\varphi,\psi\rangle_{X_{-1},X^*_1}=\langle \varphi, A^*\psi\rangle_{X}
    \;\;\mbox{for any}\;\;\varphi\in X,\;\;\psi\in X_1^*.
\end{equation*}
  Moreover, $(\rho_0I-A^*)^{-1}\in\mathcal{L}(X;X_1^*)$ and  $(\rho_0I-\widetilde{A})^{-1}\in \mathcal{L}(X_{-1};X)$ (see \cite[Proposition 2.10.3, Chapter 2]{Tucsnak-Weiss}).  Define $\widetilde{S}(\cdot):=(\rho_0I-\widetilde{A})S(\cdot)(\rho_0I-\widetilde{A})^{-1}$. Then $\widetilde{S}(\cdot)$ is  an immediately  compact semigroup on $X_{-1}$, and $\widetilde{A}$ with domain $X$ is its generator (see \cite[Proposition 2.10.4,  Chapter 2]{Tucsnak-Weiss}). We refer to $\widetilde{S}(\cdot)$ as the extension of $S(\cdot)$ from $X$ to $X_{-1}$.
 In view of  these facts, Assumption $(A_2)$ can be equivalently replaced by $(\rho_0I-\widetilde{A})^{-1}B\in
  \mathcal{L}(U;X)$. Furthermore, by  $(A_2)$, we have $B^*\in \mathcal{L}(X^*_1;U)$, i.e.,  $B^*(\rho_0I-A^*)^{-1}\in\mathcal{L}(X;U)$.
\item[$(iv)$] Assumption $(A_3)$ is equivalent to that, for any $T>0$, there exists a constant $C(T)>0$ such that  \eqref{yu-5-4-2} holds
    (see \cite[Proposition 4.3.2,  Chapter 4]{Tucsnak-Weiss}). Moreover, Assumptions $(A_2)$ and $(A_3)$ imply that
    $\int_0^\cdot \widetilde{S}(\cdot-s)Bu(s)\mathrm ds\in C([0,+\infty);X)$ for each $u\in L^2_{loc}(\mathbb{R}^+;U)$.
\item[$(v)$] Under Assumptions $(A_1)$-$(A_3)$, system \eqref{yu-11-23-1} admits a unique solution in $C([0,+\infty);X)$ for every $y_0\in X$, $u\in L_{loc}^2(\mathbb{R}^+;U)$, and $\varphi\in L^2(-\tau,0;X)$
(see, for instance, \cite[Section 2]{Nakagiri}). Throughout this paper, we denote system \eqref{yu-11-23-1} by  $[A,B]_{\kappa,\tau}$  and its solution by $y(\cdot;y_0,\varphi,u)$. When $\kappa=0$ (i.e., the delay term vanishes), we simply write $[A,B]$ instead of  $[A,B]_{0,\tau}$.
\end{enumerate}
\end{remark}

    \subsection{Main result}

   We begin this subsection by recalling the notions of exponential stabilizability and rapid stabilizability for system $[A,B]$. These definitions can be found in \cite{Liu-Wang-Xu-Yu}, and further comments are given in \cite{Kunisch-Wang-Yu, Ma-Trelat-Wang-Yu, Ma-Wang-Yu}.
\begin{definition}\label{yu-definition-10-18-1}
    \begin{enumerate}
      \item [$(i)$] Let $\alpha\in\mathbb{R}^+$. System $[A,B]$ is said to be
  $\alpha$-stabilizable  if there exists a $C_0$-semigroup $\mathcal{T}(\cdot)$ on $X$, with generator $M: D(M)\subset X\to X$, and an operator
  $K\in \mathcal{L}(D(M);U)$ such that
\begin{enumerate}
  \item [$(a)$] there exists a constant $C_1\geq 1$ such that $\|\mathcal{T}(t)\|_{\mathcal{L}(X)}
  \leq C_1e^{-\alpha t}$ for all  $t\in\mathbb{R}^+$;
\item[$(b)$] for any $x\in D(M)$,  $M x=\widetilde{A}x+BKx$;
  \item [$(c)$] there exists a constant $C_2\geq 0$ such that $\|K\mathcal{T}(\cdot)x\|_{L^2(\mathbb{R}^+;U)}
  \leq C_2\|x\|_X$ for all $x\in D(M)$.
\end{enumerate}
    The operator  $K$ is  called a feedback law associated with the $\alpha$-stabilizability of system $[A,B]$. Equivalently,
    system $[A,B]$ is said to be $\alpha$-stabilizable with the feedback law $K$.
      \item [$(ii)$] System $[A,B]$ is said to be exponentially stabilizable if it is $\alpha$-stabilizable for some $\alpha\in\mathbb{R}^+$.
If it is $\alpha$-stabilizable for every $\alpha\in\mathbb{R}^+$, then it is said to be rapidly stabilizable.
    \end{enumerate}
\end{definition}

    We now introduce  the notion of   stabilizability for system $[A,B]_{\kappa,\tau}$.
\begin{definition}\label{yu-def-12-18-1}
    Let $\kappa\in\mathbb{R}\setminus\{0\}$ and $\tau\in\mathbb{R}^+$.
\begin{enumerate}
     \item [$(i)$] Let $\alpha\in\mathbb{R}^+$. System $[A,B]_{\kappa,\tau}$ is said to be $\alpha$-stabilizable if there exists an operator
     $K:=K(\alpha)\in\mathcal{L}(X\times L^2(-\tau,0;X);U)$ and a constant $C(\alpha)>0$ such that  for each $y_0\in X$ and $\varphi\in L^2(-\tau,0;X)$, the corresponding solution $y_K(\cdot;y_0,\varphi)$
          to the closed-loop system:
\begin{equation}\label{yu-12-20-100}
\begin{cases}
    y_t(t)=Ay(t)+\kappa y(t-\tau)+B[K(y(t),
    y(t+\cdot))^\top],&t\in\mathbb{R}^+,\\
    y(0)=y_0,\\
    y(t)=\varphi(t),&t\in (-\tau,0)
\end{cases}
\end{equation}
    satisfies
\begin{equation}\label{yu-12-20-101}
    \|y_K(t;y_0,\varphi)\|_{X}\leq C(\alpha) e^{-\alpha t}(\|y_0\|^2_{X}
    +\|\varphi\|^2_{L^2(-\tau,0;X)})^{\frac{1}{2}}\;\;\mbox{for all}\;\;t\in\mathbb{R}^+.
\end{equation}
    The operator  $K$ is  called a feedback law associated with the  $\alpha$-stabilizability of system $[A,B]_{\kappa,\tau}$.
    \item[$(ii)$] Let $\alpha \in \mathbb{R}^+$. System $[A,B]_{\kappa,\tau}$ is said to be $\alpha$-stabilizable with static feedback if there exists an operator $K := K(\alpha) \in \mathcal{L}(X;U)$ and a constant $C(\alpha) > 0$ such that, for each $y_0 \in X$ and $\varphi \in L^2(-\tau,0;X)$, the corresponding solution $y_K(\cdot; y_0, \varphi)$ to the closed-loop system:
\begin{equation}\label{yu-12-20-200}
\begin{cases}
    y_t(t) = Ay(t) + \kappa y(t-\tau) + BKy(t), & t \in \mathbb{R}^+,\\
    y(0) = y_0,\\
    y(t) = \varphi(t), & t \in (-\tau,0),
\end{cases}
\end{equation}
satisfies \eqref{yu-12-20-101}.
\item[$(iii)$] System $[A,B]_{\kappa,\tau}$ is said to be stabilizable (respectively, stabilizable with static feedback) if it is $\alpha$-stabilizable (respectively, $\alpha$-stabilizable with  static feedback) for some $\alpha \in \mathbb{R}^+$.
     \item[$(iv)$] System $[A,B]_{\kappa,\tau}$ is said to be rapidly stabilizable (respectively, rapidly stabilizable with static feedback) if it is $\alpha$-stabilizable (respectively, $\alpha$-stabilizable with static feedback) for every $\alpha \in \mathbb{R}^+$.
\end{enumerate}
\end{definition}
\begin{remark}\label{yu-remark-11-23-2}
    Two remarks on Definition \ref{yu-def-12-18-1} are in order.
    \begin{enumerate}
      \item [$(i)$] System \eqref{yu-12-20-100} admits a unique solution in $C([0,+\infty);X)$ whenever $K$ is bounded (see Proposition \ref{yu-proposition-11-26-2}). Hence, the notion of $\alpha$-stabilizability for system $[A,B]_{\kappa,\tau}$ is well defined. The same holds for $\alpha$-stabilizability with  static feedback and related notations (see Remark~\ref{remark-11-27-1}).
      \item [$(ii)$] By the state-space approach (see Section \ref{sec-yu-10-1-2}), system \eqref{yu-12-20-100} can be rewritten
      in the following coupled form:
\begin{equation*}
\begin{cases}
    y_t(t)=Ay(t)+\kappa z(t,-\tau)+B[K(y(t),
    z(t,\cdot))^\top],&t\in\mathbb{R}^+,\\
    \partial_t z(t,\theta)=\partial_{\theta} z(t,\theta),
    & (t,\theta)\in \mathbb{R}^+\times(-\tau,0),\\
    z(t,0)=y(t), & t\in[0,+\infty),\\
    z(0,\theta)=\varphi(\theta), & \theta\in(-\tau,0),\\
    y(0)=y_0.
\end{cases}
\end{equation*}
     This is the closed-loop system associated with system $[A,B]_{\kappa,\tau}$ under a dynamic feedback control. In other words, the stabilizability of $[A,B]_{\kappa,\tau}$ is achieved via dynamic feedback. Since the feedback law acts on the augmented state, it generally depends on the historical behavior of the system. By contrast, the stabilizability via static feedback introduced in $(ii)$ of Definition
     \ref{yu-def-12-18-1} requires the feedback law to depend solely on the current state $y(t)$, without involving any historical information.

      \item [$(iii)$] The use of bounded feedback to stabilize a delayed control system with an unbounded control operator may seem restrictive. However, under the present assumptions, this does not lead to any essential loss of generality.
Similar results were obtained in \cite{Ma-Trelat-Wang-Yu} and are further supported by our main result (i.e., Theorem \ref{yu-theorem-9-01-1}).
    \end{enumerate}

\end{remark}

We are now in a position to state the main result of this paper.
\begin{theorem}\label{yu-theorem-9-01-1}
    Let $\kappa\in\mathbb{R}\setminus\{0\}$ and $\tau>0$. Suppose that Assumptions $(A_1)$-$(A_3)$ hold. Then the following statements are equivalent:
\begin{enumerate}
  \item[$(i)$] System $[A,B]$ is rapidly stabilizable.
  \item [$(ii)$] System $[A,B]_{\kappa,\tau}$ is rapidly stabilizable.
  \item [$(iii)$] System $[A,B]_{\kappa,\tau}$ is  rapidly stabilizable with  static feedback.
\end{enumerate}
\end{theorem}
\begin{remark}\label{yu-remark-25-9-22}
    Several remarks on Theorem \ref{yu-theorem-9-01-1} are given.
\begin{enumerate}
  \item [$(i)$] By \cite[Theorem 2]{Kunisch-Wang-Yu} and \cite[Lemma 2.4]{Ma-Trelat-Wang-Yu}, it follows that, under Assumptions $(A_1)$-$(A_3)$, system $[A,B]$ is rapidly stabilizable in the sense of Definition \ref{yu-definition-10-18-1} if and only if, for each $\alpha>0$, there exists a constant $C(\alpha)>0$ such that
\begin{equation*}\label{yu-25-7-16-1}
    \|\varphi\|_{X}^2\leq C(\alpha)(\|(\lambda I-A^*)\varphi\|_X^2+\|B^*\varphi\|_U^2)
    \;\;\mbox{for all}\;\;\lambda\in\mathbb{C}_{-\alpha}^+\;\;\mbox{and}\;\;\varphi\in X_1^*.
\end{equation*}
    Combining this characterization with Theorem \ref{yu-theorem-9-01-1}, we obtain  a frequency-domain characterization of  rapid
    stabilizability for system $[A,B]_{\kappa,\tau}$. The following weak observability inequality: for each $\alpha>0$, there exist positive constants
    $C(\alpha)$ and $D(\alpha)$ such that
$$
    \|S^*(t)\varphi\|_X^2\leq C(\alpha)\int_0^t\|B^*S^*(s)\varphi\|_U^2\mathrm ds+D(\alpha)e^{-\alpha t}\|\varphi\|_X^2\;\;\mbox{for all}\;\;\varphi\in X,\;\;t\in \mathbb{R}^+,
$$
     provides a time-domain characterization of rapid stabilizability for  system $[A,B]$ (see \cite{Liu-Wang-Xu-Yu}).  From Theorem \ref{yu-theorem-9-01-1}, it also serves as the corresponding time-domain characterization of rapid stabilizability for the delayed system $[A,B]_{\kappa,\tau}$.
  \item [$(ii)$] Theorem \ref{yu-theorem-9-01-1} shows that, under Assumptions $(A_1)$-$(A_3)$, the rapid stabilizability of  system $[A,B]_{\kappa,\tau}$ is independent of the delay length. The underlying reason is that, under Assumption $(A_1)$, for any prescribed decay rate, system $[A,B]_{\kappa,\tau}$ can be decomposed within the $C_0$-framework (introduced in Section \ref{sec-yu-10-1-2}) into an exponentially stable infinite-dimensional subsystem and a controllable finite-dimensional subsystem. Although the delay may effect the finite-dimensional part arising from this decomposition, it does not change its essential geometric structure or controllability. As a result, the stabilization mechanism remains unchanged.
       A typical example is provided by delayed heat equations, where delays may originate from sensing or communication lags. Although the delay influences the transient evolution of the temperature field, it does not alter the finite-dimensional unstable structure that determines rapid stabilizability. Therefore, the delayed system possesses the same rapid stabilizability as the corresponding delay-free system. For further details, see Section \ref{yu-sec-10-1}.
  \item [$(iii)$] One of the main objectives of this paper is to highlight the distinction between rapid stabilizability and controllability for infinite-dimensional control systems. At present, it remains unclear whether Assumption $(A_1)$ can be weakened or whether the results remain valid for time-varying delay parameters. A particularly interesting issue is whether a similar phenomenon also holds for hyperbolic control systems.
\end{enumerate}
\end{remark}
   \subsection{Novelties of this paper}

As mentioned above, time delays may severely deteriorate controllability, while their impact on rapid stabilizability has remained largely unclear. In this paper, we resolve this issue for system $[A,B]_{\kappa,\tau}$ under Assumptions $(A_1)$-$(A_3)$.
The main contributions of this paper can be summarized as follows.

First, we show that, in sharp contrast to controllability, time delays do not affect rapid stabilizability within our framework. More precisely, the rapid stabilizability of control systems with time delays is completely determined by the state and control operators and is independent of the delay parameters. This provides a complete resolution of Problem (P1).

Second, although the classical state-space approach typically leads to dynamic feedback laws, we prove that static feedback alone suffices to achieve rapid stabilization. This resolves Problem (P2) and yields a substantially simpler feedback structure.

Taken together, these results reveal several new and noteworthy features of rapid stabilizability in  control systems with time delays.
%

\subsection{Plan of this paper}
This paper is organized as follows. In Section \ref{sec-yu-10-1-2}, we focus on the well-posedness and admissibility of the control system. In Section \ref{yu-section-4}, we present the proof of the main result. Applications are provided in Section \ref{yu-sec-10-1-4}. Section \ref{yu-sec-6.2} is devoted to the appendix.

\section{Preliminaries}\label{sec-yu-10-1-2}
Fix $\tau>0$ and $\kappa\in\mathbb{R}\setminus\{0\}$.
By $(v)$ of Remark \ref{yu-remark-8-00-1}, for each $(y_0,\varphi,u)^\top\in X\times L^2(-\tau,0;X)\times L^2_{loc}(\mathbb{R}^+;U)$, system \eqref{yu-11-23-1} admits a unique solution $y(\cdot;y_0,\varphi,u)\in C([0,+\infty);X)$.
Let $(y_0,\varphi)^\top\in X\times L^2(-\tau,0;X)$ be given and write $y(\cdot):=y(\cdot;y_0,\varphi,0)$ for simplicity.

Define
\begin{equation}\label{yu-4-9-2}
z(t,\theta):=y(t+\theta), \;\; (t,\theta)\in [0,+\infty)\times(-\tau,0).
\end{equation}
Since $y(\cdot)\in C([0,+\infty);X)$, it follows that
$$
z(\cdot,\cdot)\in C([0,+\infty);L^2(-\tau,0;X)),\;\;
z(\cdot,0)\in C([0,+\infty);X),\;\;
z(\cdot,-\tau)\in L^2_{\mathrm{loc}}(\mathbb{R}^+;X).
$$
Moreover, by \eqref{yu-11-23-1} and \eqref{yu-4-9-2}, one readily verifies that $z(\cdot,\cdot)$ satisfies
\begin{equation*}\label{yu-4-9-4}
\begin{cases}
    \partial_t z(t,\theta)=\partial_{\theta} z(t,\theta),
    & (t,\theta)\in \mathbb{R}^+\times(-\tau,0),\\
    z(t,0)=y(t), & t\ge 0,\\
    z(0,\theta)=\varphi(\theta), & \theta\in(-\tau,0),
\end{cases}
\end{equation*}
and
\begin{equation*}\label{yu-4-9-5}
y_t(t)=Ay(t)+\kappa z(t,-\tau), \;\; t\in\mathbb{R}^+.
\end{equation*}

Let $\mathcal{X}:=X\times L^2(-\tau,0;X)$. Then $\mathcal{X}$ is a complex separable Hilbert space endowed with the inner product
\begin{equation}\label{yu-4-9-6}
\langle (f_1,f_2)^\top,(\varphi_1,\varphi_2)^\top\rangle_{\mathcal{X}}
=\langle f_1,\varphi_1\rangle_X+\langle f_2,\varphi_2\rangle_{L^2(-\tau,0;X)},
\quad (f_1,f_2)^\top,(\varphi_1,\varphi_2)^\top\in\mathcal{X}.
\end{equation}
Since $X=X'$, it follows that $\mathcal{X}=\mathcal{X}'$.

Define the operator family $\mathcal{S}(\cdot):[0,+\infty)\to\mathcal{L}(\mathcal{X})$ by
\begin{equation}\label{yu-4-16-1}
\mathcal{S}(t)
\begin{pmatrix}
y_0\\
\varphi
\end{pmatrix}
:=
\begin{pmatrix}
y(t)\\
z(t,\cdot)
\end{pmatrix},
\;\; t\ge 0.
\end{equation}
The following result is classical (see \cite[Proposition 3.1]{Nakagiri}).
\begin{lemma}\label{yu-lemma-4-16-1}
Suppose that Assumption $(A_1)$ holds. Then the operator family $\mathcal{S}(\cdot)$ defined by \eqref{yu-4-16-1} is an eventually compact semigroup on $\mathcal{X}$, whose generator $\mathcal{A}$ is given by
\begin{equation}\label{yu-4-9-10}
\mathcal{A}(f_1,f_2)^\top :=
\begin{pmatrix}
Af_1 + \kappa f_2(-\tau) \\
\partial_\theta f_2
\end{pmatrix},
\quad (f_1,f_2)^\top \in D(\mathcal{A}),
\end{equation}
with domain
\begin{equation}\label{yu-4-9-11}
D(\mathcal{A})
:= \left\{(f_1,f_2)^\top \in \mathcal{X}:
f_1 \in X_1, f_2 \in H^1(-\tau,0;X), f_2(0)=f_1 \right\}.
\end{equation}
Moreover, $\mathcal{S}(t)$ is compact for all $t>\tau$.
\end{lemma}

Since $\mathcal{A}$ generates a $C_0$-semigroup $\mathcal{S}(\cdot)$ on $\mathcal{X}$, it follows from \cite[Theorem 5.3, Chapter 1]{Pazy} that there exists a constant $\gamma=\gamma(\tau,\kappa,A)\in\mathbb{R}^+$ such that
\begin{equation}\label{yu-10-12-bbb-2}
(\rho I-\mathcal{A})^{-1}\in\mathcal{L}(\mathcal{X})
\;\;\mbox{for all}\;\; \rho>\gamma.
\end{equation}

By \eqref{yu-4-9-10} and \eqref{yu-4-9-11}, the adjoint operator $\mathcal{A}^*$ of $\mathcal{A}$ is given by
\begin{equation}\label{yu-4-9-12}
\mathcal{A}^*(\xi_1,\xi_2)^\top :=
\begin{pmatrix}
A^*\xi_1+\xi_2(0) \\
-\partial_\theta \xi_2
\end{pmatrix},
\quad (\xi_1,\xi_2)^\top \in D(\mathcal{A}^*),
\end{equation}
with domain
\begin{equation}\label{yu-4-12-3}
D(\mathcal{A}^*)
:= \left\{(\xi_1,\xi_2)^\top \in \mathcal{X}:
\xi_1\in X_1^*, \xi_2\in H^1(-\tau,0;X), \xi_2(-\tau)=\kappa\xi_1 \right\}.
\end{equation}

Combining Lemma \ref{yu-lemma-4-16-1} with \cite[Corollary 10.6, Chapter 1]{Pazy} and \cite[Theorem 6.4, Chapter 6]{Brezis}, we obtain the following result.

\begin{corollary}\label{wang-11-8-1}
Suppose that Assumption $(A_1)$ holds. Then $\mathcal{A}^*$ with domain $D(\mathcal{A}^*)$ generates an eventually compact semigroup $\mathcal{S}^*(\cdot)$ on $\mathcal{X}$. Moreover,
$$
\mathcal{S}^*(t)=(\mathcal{S}(t))^*\;\;\mbox{for all}\;\; t\ge 0,
$$
and $\mathcal{S}^*(t)$ is compact for all $t>\tau$.
\end{corollary}

Fix $\rho_1>\gamma$, where $\gamma\in\mathbb{R}$ is given by \eqref{yu-10-12-bbb-2}.
Let $\mathcal{X}_{-1}$ be the completion of $\mathcal{X}$ with respect to the norm
$\|f\|_{-1}:=\|(\rho_1I-\mathcal{A})^{-1}f\|_{\mathcal{X}}$ ($f\in \mathcal{X}$),
and define
\begin{equation}\label{yu-4-16-12}
\mathcal{B}:=(B,0)^\top.
\end{equation}
Following the argument in $(iii)$ of Remark \ref{yu-remark-8-00-1}, we define $\widetilde{\mathcal{A}}\in\mathcal{L}(\mathcal{X};\mathcal{X}_{-1})$ as the unique extension of $\mathcal{A}$ characterized by
\begin{equation}\label{yu-5-4-3}
\langle \widetilde{\mathcal{A}}\varphi,\psi\rangle_{\mathcal{X}_{-1},\mathcal{X}_1^*}
=\langle \varphi,\mathcal{A}^*\psi\rangle_{\mathcal{X}},
\quad \forall\, \varphi\in\mathcal{X},\;\psi\in\mathcal{X}_1^*.
\end{equation}
Here and in the following, $\mathcal{X}_1^*$ denotes the space $D(\mathcal{A}^*)$ endowed with the norm
 $\|(\rho_1I-\mathcal{A}^*)f\|_{\mathcal{X}}$ ($f\in D(\mathcal{A}^*)$).
Moreover, by arguments similar to those in $(ii)$ and $(iii)$ of Remark \ref{yu-remark-8-00-1}, the space $\mathcal{X}_1^*$ is the dual space of $\mathcal{X}_{-1}$ with respect to the pivot space $\mathcal{X}$ and the corresponding duality pairing is denoted by
$\langle\cdot,\cdot,\rangle_{\mathcal{X}_{-1},\mathcal{X}_1^*}$.
Furthermore, the operator $\widetilde{\mathcal{A}}$ with domain $\mathcal{X}$ generates an eventually compact $C_0$-semigroup $\widetilde{\mathcal{S}}(\cdot)$ on $\mathcal{X}_{-1}$, given by
$$
\widetilde{\mathcal{S}}(t):=(\rho_1 I-\widetilde{\mathcal{A}})\,\mathcal{S}(t)\,(\rho_1 I-\widetilde{\mathcal{A}})^{-1},
\;\; t\ge 0.
$$

The next two lemmas establish the admissibility of the control operator $\mathcal{B}$.
\begin{lemma}\label{yu-lemma-10-13-1}
    Suppose that Assumptions $(A_1)$ and $(A_2)$ hold. Then
    $\mathcal{B}\in\mathcal{L}(U;\mathcal{X}_{-1})$.
\end{lemma}
\begin{proof}
Since $(\rho_1 I-\widetilde{\mathcal{A}})^{-1}$ is an isometric isomorphism from $\mathcal{X}_{-1}$ to $\mathcal{X}$ (see \cite[Proposition 2.10.3, Chapter 2]{Tucsnak-Weiss}), it suffices to prove that
$(\rho_1 I-\widetilde{\mathcal{A}})^{-1}\mathcal{B}\in \mathcal{L}(U;\mathcal{X})$.
Let $g\in U$. By \eqref{yu-4-16-12} and Assumption $(A_2)$, $\mathcal{B}g=(Bg,0)^\top\in \mathcal{X}_{-1}$. For any $f\in\mathcal{X}$, let $\psi^f:=(\psi^f_1,\psi^f_2)^\top\in \mathcal{X}_1^*$ be the unique solution to
$(\rho_1 I-\mathcal{A}^*)\psi^f=f$.
Then, by \eqref{yu-10-12-bbb-2}, there exists a constant $C>0$ such that
$\|\psi_1^f\|_{X_1^*}+\|\psi_2^f\|_{H^1(-\tau,0;X)}\le C\|f\|_{\mathcal{X}}$.
This, together with Assumption $(A_2)$, yields that
$$
|\langle Bg,\psi_1^f\rangle_{X_{-1},X_1^*}|
\le C\|g\|_U\|f\|_{\mathcal{X}}\;\;\mbox{for all}\;\; f\in\mathcal{X}.
$$
Hence, by the Riesz representation theorem, there exists a unique $\varphi\in\mathcal{X}$ such that
$\langle \varphi,f\rangle_{\mathcal{X}}=\langle Bg,\psi_1^f\rangle_{X_{-1},X_1^*}$ for all $f\in \mathcal{X}$,
and  $\|\varphi\|_{\mathcal{X}}\le C\|g\|_U$.
By \eqref{yu-5-4-3} and the definition of $\psi^f$, we deduce that
$
\varphi=(\rho_1 I-\widetilde{\mathcal{A}})^{-1}\mathcal{B}g$.
Thus $(\rho_1 I-\widetilde{\mathcal{A}})^{-1}\mathcal{B}\in \mathcal{L}(U;\mathcal{X})$, and the proof is complete.
\end{proof}

\begin{lemma}\label{yu-lemma-10-13-2}
    Suppose that Assumptions $(A_1)$-$(A_3)$ hold. Then, for each $T>0$, there exists a constant $C(T)>0$ such that
\begin{equation}\label{yu-10-8-1}
    \int_0^T\|\mathcal{B}^*\mathcal{S}^*(t)\varphi\|_{U}^2\mathrm dt
    \leq C(T)\|\varphi\|^2_{\mathcal{X}}\;\;\mbox{for all}\;\;\varphi\in \mathcal{X}^*_1.
\end{equation}
\end{lemma}
\begin{proof}
It suffices to prove \eqref{yu-10-8-1} for $T\in(0,\tau)$ (see \cite[Proposition 4.3.2]{Tucsnak-Weiss}). Fix such a $T$.

\vskip 5pt
\emph{Step 1.} Let $\varphi=(\varphi_1,\varphi_2)^\top\in \mathcal{X}_1^*$ with $\varphi_2(\theta)\in X_1^*$ for a.e. $\theta\in(-\tau,0)$. Denote $(\xi(t),\eta(t,\cdot))^\top=\mathcal{S}^*(t)\varphi$. Then $\xi$ satisfies
$$
\xi(t)=e^{A^*t}\varphi_1+\int_0^t e^{A^*(t-s)}\varphi_2(-s)\mathrm ds,\quad t\in[0,T].
$$
By Assumption $(A_3)$ and $(iv)$ of Remark \ref{yu-remark-8-00-1}, we obtain
$$
\int_0^T\|B^*\xi(t)\|_U^2\mathrm dt
\le C(T)\bigl(\|\varphi_1\|_X^2+\|\varphi_2\|_{L^2(-\tau,0;X)}^2\bigr)
= C(T)\|\varphi\|_{\mathcal{X}}^2.
$$
This, together with \eqref{yu-4-16-12}, yields \eqref{yu-10-8-1} for such $\varphi$.

\vskip 5pt
\emph{Step 2.} For general $\varphi\in \mathcal{X}_1^*$, take $\mathcal{R}_n=\mathrm{diag}(n(nI-A^*)^{-1},\,n(nI-A^*)^{-1})$. Then $\mathcal{R}_n\varphi\in \mathcal{X}_1^*$ and $[\mathcal{R}_n\varphi]_2(\theta)\in X_1^*$ for a.e. $\theta\in(-\tau,0)$. By \emph{Step 1},
$\int_0^T\|\mathcal{B}^*\mathcal{S}^*(t)\mathcal{R}_n\varphi\|_U^2\mathrm dt
\le C(T)\|\mathcal{R}_n\varphi\|_{\mathcal{X}}^2$.
Since $\mathcal{R}_n\varphi\to\varphi$ in $\mathcal{X}_1^*$ (see \cite[Theorem 3.1 and Lemma 3.2, Chapter 1]{Pazy}), a standard density argument yields \eqref{yu-10-8-1} for all $\varphi\in \mathcal{X}_1^*$.
\end{proof}

By Lemmas~\ref{yu-lemma-4-16-1}, \ref{yu-lemma-10-13-1} and \ref{yu-lemma-10-13-2}, the pair $(\mathcal{A},\mathcal{B})$ fits into the abstract framework studied in \cite{Kunisch-Wang-Yu}.
Consequently, for each $Y_0\in\mathcal{X}$ and $u\in L^2_{loc}(\mathbb{R}^+;U)$, the system
\begin{equation}\label{yu-10-14-2}
\begin{cases}
    Y_t(t)=\mathcal{A}Y(t)+\mathcal{B}u(t), & t\in\mathbb{R}^+,\\
    Y(0)=Y_0,
\end{cases}
\end{equation}
admits a unique mild solution $Y(\cdot;Y_0,u)\in C([0,+\infty);\mathcal{X})$.

The following proposition clarifies the relationship between \eqref{yu-11-23-1} and \eqref{yu-10-14-2}.
Although this result is well known, we include a proof for completeness.
\begin{proposition}\label{yu-proposition-10-14-1}
     Suppose that Assumptions $(A_1)$-$(A_3)$ hold. Let $(y_0,\varphi,u)\in X\times L^2(-\tau,0;X)\times L_{loc}^2(\mathbb{R}^+;U)$.
    The following statements hold:
\begin{enumerate}
     \item [$(i)$] If $y(\cdot;y_0,\varphi,u)$
         is the solution to \eqref{yu-11-23-1}, then
                            $Y(\cdot):=(y_1(\cdot),y_2(\cdot))^{\top}$ is the solution to
  \eqref{yu-10-14-2} with $Y_0:=(y_0,\varphi)^{\top}$, where $y_1(t):=y(t;y_0,\varphi,u)$
  and $[y_2(t)](\cdot):=y(t+\cdot;y_0,\varphi,u)$ for all $t\in [0,+\infty)$.
    \item [$(ii)$] If $Y(\cdot;Y_0,u)=(y_1(\cdot;Y_0,u),y_2(\cdot;Y_0,u))^\top$ (with
     $Y_0:=(y_0,\varphi)^\top$)
          is the solution to \eqref{yu-10-14-2}, then the function $y(\cdot)$ defined by
\begin{equation}\label{yu-12-20-8}
    y(t):=
\begin{cases}
    y_1(t;Y_0,u),&t\in [0,+\infty),\\
    \varphi(t),&t\in(-\tau,0),
\end{cases}
\end{equation}
   is the solution to \eqref{yu-11-23-1}. Moreover,
$[y_2(t;Y_0,u)](\theta)=y(t+\theta)$ for all $t\in [0,+\infty)$ and a.e. $\theta\in
  (-\tau,0)$.
  \end{enumerate}
\end{proposition}
\begin{proof}
We first prove $(i)$. Since $y(\cdot;y_0,\varphi,u)\in C([0,+\infty);X)$ and coincides with $\varphi$ on $(-\tau,0)$, it follows that $Y(0)=Y_0$ and $Y(\cdot)\in C([0,+\infty);\mathcal{X})$.
Let $\xi=(\xi_1,\xi_2)^\top\in \mathcal{X}_1^*$ be arbitrary. On the one hand, by the equation satisfied by $y(\cdot;y_0,\varphi,u)$, we have
$$
\frac{d}{dt}\langle y_1(t),\xi_1\rangle_X
=\langle y_1(t),A^*\xi_1\rangle_X+\langle y_1(t-\tau),\kappa\xi_1\rangle_X
+\langle Bu(t),\xi_1\rangle_{X_{-1},X_1^*}.
$$
On the other hand, by a change of variables and integration by parts,
$$
\frac{d}{dt}\langle y_2(t),\xi_2\rangle_{L^2(-\tau,0;X)}
=-\langle y_2(t),\partial_\theta\xi_2\rangle_{L^2(-\tau,0;X)}
+\langle y_1(t),\xi_2(0)\rangle_X-\langle y_1(t-\tau),\xi_2(-\tau)\rangle_X.
$$
Using $\xi_2(-\tau)=\kappa\xi_1$ (see \eqref{yu-4-12-3}), we obtain
$
\frac{d}{dt}\langle Y(t),\xi\rangle_{\mathcal{X}}
=\langle Y(t),\mathcal{A}^*\xi\rangle_{\mathcal{X}}
+\langle Bu(t),\xi_1\rangle_{X_{-1},X_1^*}$.
However, by Lemma \ref{yu-lemma-10-13-1}, we have
$\langle Bu(t),\xi_1\rangle_{X_{-1},X_1^*}=\langle \mathcal{B}u(t),\xi\rangle_{\mathcal{X}_{-1},\mathcal{X}_1^*}$.
Hence by the main theorem in \cite{Ball}, $Y(\cdot)$ is the (unique) mild solution to \eqref{yu-10-14-2}.

\par
 We now show $(ii)$. Let $Y(\cdot;Y_0,u)$ be the solution to \eqref{yu-10-14-2} and define $y(\cdot)$ by \eqref{yu-12-20-8}. Then, it is clear that $y(0)=y_0$ and $y=\varphi$ on $(-\tau,0)$.
Denote $z_1(t)=y(t)$ and $z_2(t,\theta)=[y_2(t;Y_0,u)](\theta)$. Arguing as in the proof of $(i)$, for any $\xi=(\xi_1,\xi_2)^\top\in \mathcal{X}_1^*$,
\begin{eqnarray*}
&\;&\frac{d}{dt}\langle z_1(t),\xi_1\rangle_X
+\frac{d}{dt}\langle z_2(t),\xi_2\rangle_{L^2(-\tau,0;X)}\nonumber\\
&=&\langle z_1(t),A^*\xi_1+\xi_2(0)\rangle_X
-\langle z_2(t),\partial_\theta\xi_2\rangle_{L^2(-\tau,0;X)}
+\langle Bu(t),\xi_1\rangle_{X_{-1},X_1^*}.
\end{eqnarray*}
Choosing test functions of the form $(0,\xi_2^*)^\top$ and using standard arguments for transport equations (see \cite[Section 2.3.3.1, Chapter 2]{Coron}), we deduce that $z_2$ satisfies
$$
\partial_t z_2=\partial_\theta z_2,\quad z_2(t,0)=z_1(t),\quad z_2(0,\theta)=\varphi(\theta), \;\;\theta\in(-\tau,0),
$$
which yields $z_2(t,\theta)=z_1(t+\theta)$ for a.e. $\theta\in(-\tau,0)$.
Substituting this relation into the previous identity gives
$$
\frac{d}{dt}\langle y(t),\xi_1\rangle_X
=\langle y(t),A^*\xi_1\rangle_X+\langle \kappa y(t-\tau),\xi_1\rangle_X
+\langle Bu(t),\xi_1\rangle_{X_{-1},X_1^*}.
$$
Hence $y(\cdot)$ solves \eqref{yu-11-23-1}. The relation
$[y_2(t;Y_0,u)](\theta)=y(t+\theta)$ for a.e. $\theta\in(-\tau,0)$
follows immediately.
The proof is complete.
\end{proof}

  At the end of this section, we present a frequency-domain characterization of the rapid stabilizability of system \eqref{yu-10-14-2}. This result follows from Lemmas \ref{yu-lemma-4-16-1}, \ref{yu-lemma-10-13-1}, \ref{yu-lemma-10-13-2},  and \cite[Lemma 3.1]{Ma-Trelat-Wang-Yu}.
 \begin{proposition}\label{yu-lemma-1-12-1}
    Suppose that Assumptions $(A_1)$-$(A_3)$ hold. Then the following statements are equivalent:
\begin{enumerate}
  \item [$(i)$] System \eqref{yu-10-14-2} is rapidly stabilizable.
  \item [$(ii)$] For each $\alpha>0$, there exists a constant $C(\alpha)>0$ such that
\begin{equation*}\label{yu-1-4-11}
    \|\varphi\|^2_{\mathcal{X}}\leq C(\alpha)(\|(\lambda I-\mathcal{A}^*)\varphi\|_{\mathcal{X}}^2
    +\|\mathcal{B}^*\varphi\|_U^2)
    \;\;\mbox{for all}\;\;\lambda\in\mathbb{C}_{-\alpha}^+\;\;\mbox{and}\;\;\varphi\in \mathcal{X}_{1}^*.
\end{equation*}
\end{enumerate}
   \end{proposition}

\section{Proof of main theorem}\label{yu-section-4}
    We start this section with a lemma that plays a key role in the proof of Theorem \ref{yu-theorem-9-01-1}.

\begin{lemma}\label{yu-lemma-11-29-1}
For each $\kappa\in\mathbb{R}$ and $\tau>0$, one has
\begin{equation*}
\mathbb{C}=\{z-\kappa e^{-z\tau}:z\in\mathbb{C}\}.
\end{equation*}
\end{lemma}

\begin{proof}
Fix $\kappa\in\mathbb{R}$ and $\tau>0$.
Define the entire function:
$$
F(z):=z-\kappa e^{-z\tau},\;\; z\in\mathbb{C}.
$$
 It suffices to show that $F(\cdot)$ is surjective.
\par
When $\kappa=0$, it is trivial. When $\kappa\neq 0$,
let $\eta\in\mathbb{C}$ be arbitrarily fixed. We need to show the existence of $z\in\mathbb{C}$ such that
$F(z)=\eta$, i.e.,
\begin{equation}\label{yu-26-6-24-1}
(z-\eta)e^{z\tau}=\kappa.
\end{equation}
Since $\kappa\neq 0$, it is clear that $z=\eta$ does not satisfy the equation \eqref{yu-26-6-24-1}. Therefore, for $z\in\mathbb{C}\setminus\{\eta\}$, we set
$w:=\frac{1}{\tau(z-\eta)}$ and observe that \eqref{yu-26-6-24-1} is equivalent to
\begin{equation}\label{yu-26-6-21-1}
w^{-1}e^{w^{-1}}=\kappa\tau e^{-\tau\eta}.
\end{equation}
\par
    Define
\begin{equation}\label{yu-26-6-21-2}
    G(w):=  w^{-1}e^{w^{-1}},\;\;w\in\mathbb{C}\setminus\{0\}.
\end{equation}
Then \eqref{yu-26-6-21-1} can be rewritten as
$$
G(w)=\kappa\tau e^{-\tau\eta}.
$$
   Since $G(\cdot)$ is holomorphic on $\mathbb{C}\setminus\{0\}$ and $0$ is the essential singularity of $G(\cdot)$, it follows from the Big Picard Theorem (see \cite[Theorem 3, Chapter 4]{Narasimhan-Nievergrlt}) that $G(\mathbb{C}\setminus\{0\})$ contains all complex numbers with at most one exception.
Note that  $0\notin G(\mathbb{C}\setminus\{0\})$. Thus, $0$ is the only exception of $G(\cdot)$ and $\mathbb{C}\setminus\{0\}=G(\mathbb{C}\setminus\{0\})$, which, together with \eqref{yu-26-6-21-2} and the assumption $\kappa\neq 0$, implies the existence of $w\in\mathbb{C}\setminus\{0\}$ satisfying \eqref{yu-26-6-21-1}.
 This completes the proof.
\end{proof}

    We now proceed to prove Theorem \ref{yu-theorem-9-01-1}.
\begin{proof}[Proof of Theorem \ref{yu-theorem-9-01-1}]
    The proof is divided into three steps. To facilitate tracking the dependence and interplay of the various constants arising in the estimates, we label these constants systematically throughout the proof.
\vskip 5pt
    \emph{Step 1. $(i)\Rightarrow (iii)$.} Since system $[A,B]$ is rapidly stabilizable, it follows from  Remark \ref{yu-remark-25-9-22} that for each $\gamma>0$, there exists a constant $C_1(\gamma)>0$ such that
  \begin{equation}\label{yu-9-11-1}
    \|\varphi\|_{X}^{2}\leq C_1(\gamma)(\|(\lambda I-A^*)\varphi\|^2_{X}+\|B^*\varphi\|^2_{U})
    \;\;\mbox{for all}\;\;\lambda\in \mathbb{C}_{-\gamma}^+\;\;\mbox{and}\;\;\varphi\in X_1^*.
  \end{equation}
     Fix $\gamma>0$ and $\alpha>0$ arbitrarily. Let $y_0\in X$, $\varphi\in L^2(-\tau,0;X)$ and $u\in L^2(\mathbb{R}^+;U)$
     be given. We denote by $y(\cdot;y_0,\varphi,u,\gamma)$ the solution of  the following equation:
\begin{equation*}\label{yu-11-24-2}
\begin{cases}
    y_t(t)=(A+\gamma I)y(t)+\kappa e^{\gamma\tau}y(t-\tau)+Bu(t),&t\in\mathbb{R}^+,\\
    y(0)=y_0,\\
    y(t)=\varphi(t),&t\in(-\tau,0).
\end{cases}
\end{equation*}
     Then $y(\cdot;y_0,\varphi,u,\gamma)\in C([0,+\infty);X)$ (see
     $(v)$ in Remark \ref{yu-remark-8-00-1}) and satisfies
  \begin{equation*}\label{yu-11-24-1}
  \begin{cases}
  y(t;y_0,\varphi,u,\gamma)=S_\gamma(t)y_0+\kappa e^{\gamma\tau}\int_0^tS_\gamma(t-s)y(s-\tau)\mathrm ds
  +\int_0^t\widetilde{S}_\gamma(t-s)Bu(s)\mathrm ds ,&t\in [0,+\infty),\\
  y(t;y_0,\varphi,u,\gamma)=\varphi(t),&t\in(-\tau,0).
  \end{cases}
  \end{equation*}
     Here and in the remainder of the proof, we set $S_\gamma(\cdot):=e^{\gamma\cdot}S(\cdot)$, which is the $C_0$-semigroup on $X$ generated by $A_\gamma:=A+\gamma I$ with domain $X_1$, and  $\widetilde{S}_\gamma(\cdot):=e^{\gamma\cdot}\widetilde{S}(\cdot)$, which is the extension of $S_\gamma(\cdot)$ from $X$ to $X_{-1}$ (see $(iii)$ in Remark \ref{yu-remark-8-00-1}). Let $\beta$ be a constant satisfying
\begin{equation}\label{yu-9-20-bbb-1}
    \beta>2(\alpha+|\kappa| e^{(\gamma+\alpha)\tau}).
\end{equation}
 By $(A_1)$, the operator $A_\gamma$ generates an immediately compact semigroup on $X$. Hence, by \cite[Lemma 6.2]{Ma-Trelat-Wang-Yu},
 there exists a projection operator $P_\beta\in \mathcal{L}(X)$
    such that: $(b_1)$ $P_\beta X^*_1\subset X^*_1$ and $A^*_\gamma P_\beta=P_\beta A^*_\gamma$;
  $(b_2)$ $\sigma((P_\beta A^*_\gamma)|_{P_\beta X})\subset \mathbb{C}_{-\beta}^+$;
  $(b_3)$ $P_\beta X$ is finite dimensional;
  $(b_4)$  for each $\delta\in(0,\beta)$, there exists a constant $C_2(\beta,\delta)>0$ such that $\|(I-P_\beta)S^*_\gamma(t)\|_{\mathcal{L}(X)}\leq C_2(\beta,\delta)e^{-\delta t}$ for all $t\in\mathbb{R}^+$.

    We now present several important facts derived from the operator $P_\beta$ and conditions $(b_1)$-$(b_4)$ that will be used in the proof.
These facts enable us to decompose the system into a finite-dimensional unstable component and an exponentially stable infinite-dimensional component. This decomposition is fundamental for the construction of a stabilizing feedback law.
\par
  Firstly, define $Q_1:=P_\beta X$, $Q_2:=(I-P_\beta)X$,  $S^*_{\gamma,1}(\cdot):=P_\beta S^*_\gamma(\cdot)$, and $S^*_{\gamma,2}(\cdot):=(I-P_\beta)S^*_\gamma(\cdot)$. By $(b_1)$-$(b_4)$
  and the arguments in \cite[Lemma 6.3]{Ma-Trelat-Wang-Yu}, the following statements hold: $(c_1)$ $Q_1\subset X_1^*$ and is finite-dimensional;
      $(c_2)$ $A^*_\gamma|_{Q_1}$ (the restriction of $A^*_\gamma$ to $Q_1$) is bounded and
      $\sigma(A^*_\gamma|_{Q_1})\subset\mathbb{C}_{-\beta}^+$;
      $(c_3)$ $Q_1$ and $Q_2$ are  invariant subspaces of $S^*_\gamma(\cdot)$.
   By the compactness of $S^*_\gamma(\cdot)$,
  it follows that $S^*_{\gamma,1}(\cdot)$ and $S^*_{\gamma,2}(\cdot)$  are  immediately compact semigroups on $Q_1$ and $Q_2$, respectively. Moreover, from $(c_1)$, we deduce that
\begin{equation}\label{yu-11-24-4}
    P_\beta X^*_1=Q_1.
\end{equation}
 Let $A_{\gamma,1}^*$  and $A_{\gamma,2}^*$  be the generators of $S^*_{\gamma,1}(\cdot)$ on $Q_1$ and $S^*_{\gamma,2}(\cdot)$ on $Q_2$, respectively. We claim that
\begin{equation}\label{yu-11-24-5}
\begin{cases}
    A^*_{\gamma,1}=P_\beta A^*_\gamma\;\;\mbox{in}\;\;Q_1\;\;\mbox{with}
    \;\;D(A^*_{\gamma,1})=Q_1,\\
    A^*_{\gamma,2}=(I-P_\beta)A^*_\gamma\;\;\mbox{in}\;\;Q_2\;\;\mbox{with}
    \;\;D(A^*_{\gamma,2})=(I-P_\beta)X^*_1.
\end{cases}
\end{equation}
    On the one hand, let $f\in (I-P_\beta)X^*_1$, i.e., $f=(I-P_\beta)g$ for some $g\in X^*_1$. Then by $(b_1)$,
\begin{equation}\label{yu-11-11-4}
    \lim_{t\to 0^+}t^{-1}(S_{\gamma,2}^*(t)f-f)=(I-P_\beta)\lim_{t\to 0^+}
    t^{-1}(S^*_\gamma(t)g-g)=(I-P_\beta)A_\gamma^*g\in X,
\end{equation}
   which shows that $f\in D(A^*_{\gamma,2})$. Hence, $(I-P_\beta)X_1^*\subset D(A^*_{\gamma,2})$.
    On the other hand, let $f\in D(A^*_{\gamma,2})$. Then $f=(I-P_\beta)f\in Q_2$. Using $(b_1)$ and $(c_3)$, we have
\begin{equation*}\label{yu-11-11-5}
    \lim_{t\to 0^+}t^{-1}[S^*_\gamma(t)(I-P_\beta)f-(I-P_\beta)f]
    =\lim_{t\to0^+}t^{-1}[(I-P_\beta)S^*_\gamma(t)f-f]=\lim_{t\to0^+}t^{-1}(S^*_{\gamma,2}(t)f-f)=A^*_{\gamma,2}f\in X.
\end{equation*}
    This implies that $f=(I-P_\beta)f\in X_1^*$, and hence $D(A_{\gamma,2}^*)\subset (I-P_\beta)X^*_1$. Therefore, $D(A^*_{\gamma,2})=(I-P_\beta)X^*_1$. This, along with
    \eqref{yu-11-11-4}, $(b_1)$ and the fact $(I-P_\beta)^2=I-P_\beta$, yields that $A^*_{\gamma,2}=(I-P_\beta)A_\gamma^*$.
   By analogous arguments and \eqref{yu-11-24-4}, we obtain that $A^*_{\gamma,1}=P_\beta A^*_\gamma$ and $D(A_{\gamma,1}^*)=Q_1$.
   Hence \eqref{yu-11-24-5} follows.
\par
   Secondly, define $H_1:=P_\beta^*X$, $H_2:=(I-P_\beta^*)X$, $S_{\gamma,1}(\cdot):=P^*_\beta S_\gamma(\cdot)$, and $S_{\gamma,2}(\cdot):=(I-P^*_\beta)S_\gamma(\cdot)$, where $P^*_\beta$ denotes the adjoint of $P_\beta$. It is straightforward to verify that $P^*_\beta$ is a projection operator and that
   \begin{equation}\label{yu-1-6-20}
    S_{\gamma,1}(\cdot)=(S_{\gamma,1}^*(\cdot))^*
    \;\;\mbox{and}\;\;S_{\gamma,2}(\cdot)=(S_{\gamma,2}^*(\cdot))^*.
   \end{equation}
   We claim that
\begin{equation}\label{yu-1-6-21}
    H_1=Q_1'\;\;\mbox{and}\;\;H_2=Q_2'
\end{equation}
    in the following sense: for $i\in\{1,2\}$ and every bounded linear functional $F_i$ on $Q_i$, there exists a unique $f_i\in H_i$ such
   that $F_i(\varphi)=\langle f_i,\varphi\rangle_X$ for all $\varphi\in Q_i$.
    Indeed, let $F_1$ be a bounded linear functional on $Q_1$. Then $F_1\circ P_\beta$ defines a bounded linear functional on $X$.
    Since $X'=X$, there exists a unique $f_1\in X$ such that $F_1(P_\beta \varphi)=\langle f_1, \varphi\rangle_X$ for all $\varphi\in X$.
     Noting that  $F_1(P_\beta \varphi)=F_1(P_\beta(P_\beta \varphi))$, we deduce that $\langle f_1,\varphi\rangle_X=\langle f_1,P_\beta \varphi\rangle_X=\langle P^*_\beta f_1,\varphi\rangle_X$ for all $\varphi\in X$. Hence, $f_1=P^*_\beta f_1\in H_1$,
      which shows  that $H_1=Q_1'$. The identity for $H_2$ can be established in the same way.
\par

Thirdly, by the compactness of $S^*_{\gamma,1}(\cdot)$ on $Q_1$ and $S^*_{\gamma,2}(\cdot)$ on $Q_2$, \eqref{yu-1-6-20} and \eqref{yu-1-6-21}, we deduce that $S_{\gamma,1}(\cdot)$ and $S_{\gamma,2}(\cdot)$ are compact semigroups on $H_1$ and $H_2$, respectively. Since $Q_1$ is finite-dimensional (see $(c_1)$ above), it follows from  \eqref{yu-1-6-21} that
\begin{equation}\label{yu-12-18-1}
    \mbox{Dim}(H_1)=\mbox{Dim}(Q_1)=N<+\infty.
\end{equation}
  We claim that
\begin{equation}\label{yu-11-24-6}
    H_1\subset X_1\;\;\mbox{and}\;\;P^*_\beta X_1=H_1.
\end{equation}
   To this end, by $(b_1)$ and the first identity in \eqref{yu-1-6-20}, \cite[Theorem 5.3, Chapter 1]{Pazy}, and by applying the Laplace transform to $S_\gamma(\cdot)P_\beta^*\varphi$ ($\varphi\in X$), there exists a constant $\omega>0$ such that $\mathbb{C}_\omega^+\subset
    \rho(A_\gamma)$ and
\begin{equation}\label{yu-1-6-22}
    P_\beta^*(\lambda I-A_\gamma)^{-1}=(\lambda I-A_\gamma)^{-1}P_\beta^*
    \;\;\mbox{for all}\;\;\lambda\in \mathbb{C}_\omega^+.
\end{equation}
   Fix $\lambda_0\in \mathbb{C}_\omega^+$. By \eqref{yu-12-18-1}, let $\{e_j\}_{j=1}^N$  be an orthogonal basis of $H_1$. Define $f_j:= (\lambda_0I-A_\gamma)^{-1}e_j$ for each $j\in\{1,2,\ldots,N\}$.
   Then by \eqref{yu-1-6-22}, we have
  $\{f_j\}_{j=1}^N$ is also a basis of $H_1$. Consequently, $H_1\subset X_1$, which implies $H_1\subset P_\beta^*X_1$. Since the reverse inclusion $ P_\beta^*X_1\subset H_1$ is immediate, we obtain \eqref{yu-11-24-6}.
\par
    Finally, let $A_{\gamma,1}$  and $A_{\gamma,2}$  be the generators of $S_{\gamma,1}(\cdot)$ on $H_1$ and  $S_{\gamma,2}(\cdot)$ on $H_2$, respectively.
    By the second identity in \eqref{yu-1-6-20}, \eqref{yu-11-24-6} and the arguments analogous to those used in the proof of \eqref{yu-11-24-5}, we obtain
 \begin{equation*}
\begin{cases}
    A_{\gamma,1}=P^*_\beta A_\gamma\;\;\mbox{in}\;\;H_1\;\;\mbox{with}\;\;D(A_{\gamma,1})=H_1,\\
    A_{\gamma,2}=(I-P_\beta^*)A_\gamma\;\;\mbox{in}\;\;H_2
    \;\;\mbox{with}\;\;D(A_{\gamma,2})=(I-P_\beta^*)X_1.
\end{cases}
\end{equation*}
\par
Define $w(\cdot;u):=P^*_{\beta}y(\cdot;y_0,\varphi,u,\gamma)$
    and $v(\cdot;u):=(I-P^*_{\beta})y(\cdot;y_0,\varphi,u,\gamma)$. Introduce
\begin{equation}\label{yu-12-19-1}
    \mathcal{P}:=\mbox{diag}(P^*_\beta,P^*_\beta).
\end{equation}
    Let $\mathcal{H}_1:=H_1\times L^2(-\tau,0;H_1)$  and $\mathcal{H}_2:=H_2\times L^2(-\tau,0;H_2)$. Then $\mathcal{P}$ is a projection operator from $\mathcal{X}:=X\times L^2(-\tau,0;X)$ onto $\mathcal{H}_1$, $\mathcal{X}= \mathcal{H}_1\oplus \mathcal{H}_2$, and  $(I-\mathcal{P})\mathcal{X}=\mathcal{H}_2$. Here and in the sequel, the inner product on  $\mathcal{X}$ is defined by
    \eqref{yu-4-9-6}.
\par
    The remainder of this step is divided into the following three substeps.
\vskip 5pt
    \emph{Sub-step 1.1. We show that there exists an eventually compact semigroup $\mathcal{S}_{\gamma,2}(\cdot)$ on $\mathcal{H}_2$ and a constant $C_3(\alpha)>0$ such that
    $(v(t;0),v(t+\cdot;0))^\top=\mathcal{S}_{\gamma,2}(t)
    (I-\mathcal{P})(y_0,\varphi)^\top$ and
\begin{equation}\label{yu-9-18-1}
    \|\mathcal{S}_{\gamma,2}(t)(I-\mathcal{P})(y_0,\varphi)^\top
    \|_{\mathcal{X}}\leq C_3(\alpha)e^{-\alpha t}(\|y_0\|^2_{X}
    +\|\varphi\|^2_{L^2(-\tau,0;X)})^{\frac{1}{2}}\;\;\mbox{for all}\;\;t\in\mathbb{R}^+.
\end{equation}}
Indeed, by the definition of $v(\cdot;\cdot)$, it is clear that $v(\cdot;0)$ is the solution of the equation (in $H_2$):
\begin{equation*}\label{yu-9-13-11}
\begin{cases}
    v_t(t)=A_{\gamma,2} v(t)+\kappa e^{\gamma\tau}v(t-\tau),&t\in\mathbb{R}^+,\\
    v(0)=(I-P^*_{\beta})y_0,\\
    v(t)=(I-P^*_{\beta})\varphi(t),&t\in(-\tau,0).
\end{cases}
\end{equation*}
    On $\mathcal{H}_2$, define
\begin{equation*}\label{yu-9-18-2}
    \mathcal{A}_{\gamma,2}(f_1,f_2)^\top:=\left(
                                 \begin{array}{c}
                                   A_{\gamma,2}f_1 + \kappa e^{\gamma\tau}f_2(-\tau) \\
                                   \partial_\theta f_2 \\
                                 \end{array}
                               \right),
                               \;\;(f_1,f_2)^{\top}\in D(\mathcal{A}_{\gamma,2}),
\end{equation*}
    where
$D(\mathcal{A}_{\gamma,2}):=\{(f_1,f_2)^\top \in \mathcal{H}_2:f_1\in D(A_{\gamma,2}),
    f_2\in H^1(-\tau,0;H_2), f_2(0)=f_1\}$.
    Since $A_{\gamma,2}$ generates an immediately compact semigroup $S_{\gamma,2}(\cdot)$ on $H_2$, it follows from Lemma \ref{yu-lemma-4-16-1} that
    $\mathcal{A}_{\gamma,2}$ with domain $D(\mathcal{A}_{\gamma,2})$ generates an eventually compact semigroup
    $\mathcal{S}_{\gamma,2}(\cdot)$ on $\mathcal{H}_2$. Let
    $V(t):=(v(t;0),v(t+\cdot;0))$ for $t\in \mathbb{R}^+$. By Proposition \ref{yu-proposition-10-14-1}, we have
    $V(t)=\mathcal{S}_{\gamma,2}(t)((I-P^*_\beta) y_0,(I-P^*_\beta)\varphi)^\top$ for all $t\in\mathbb{R}^+$.
    We next show that
\begin{equation}\label{yu-9-18-4}
    \sigma(\mathcal{A}_{\gamma,2})\subset \mathbb{C}_{-\alpha}^-.
\end{equation}
Once \eqref{yu-9-18-4} is established, it follows from \cite[Corollary 3.12, Chapter 4 and Corollary 3.2, Chapter 5]{Engel-Nagel} that $\sup\{\mbox{Re}\lambda:\lambda\in \sigma(\mathcal{A}_{\gamma,2})\}<-\alpha$ and there exists a constant $C_4(\alpha)>0$ such that
\begin{equation*}\label{yu-9-20-2}
    \|V(t)\|_{\mathcal{X}}=\|V(t)\|_{\mathcal{H}_2}\leq C_4(\alpha)e^{-\alpha t}\|((I-P^*_\beta) y_0,(I-P^*_\beta)\varphi)^\top\|_{\mathcal{X}}
    \;\;\mbox{for all}\;\;t\in\mathbb{R}^+,
\end{equation*}
which yields \eqref{yu-9-18-1} with $C_3(\alpha):=C_4(\alpha)$. Observe that \eqref{yu-9-18-4} is equivalent to
\begin{equation}\label{wang-10-29-3}
(\lambda I-\mathcal{A}_{\gamma,2})\eta\not=0\;\;\mbox{for all}\;\;\lambda\in \overline{\mathbb{C}_{-\alpha}^+}\;\;\mbox{and a non-zero vector}\;\;
\eta:=(\eta_1,\eta_2)^\top\in D(\mathcal{A}_{\gamma,2}).
\end{equation}
 Thus, we only need to show \eqref{wang-10-29-3}. We prove \eqref{wang-10-29-3} by contradiction. Suppose that  there exist  $\lambda\in \overline{\mathbb{C}_{-\alpha}^+}$ and a non-zero vector $\eta:=(\eta_1,\eta_2)^\top\in D(\mathcal{A}_{\gamma,2})$ such that $(\lambda I-\mathcal{A}_{\gamma,2})\eta=0$. Then
\begin{equation*}
    (\lambda I-A_{\gamma,2})\eta_1-\kappa e^{\gamma\tau}\eta_2(-\tau)=0,\;\;
\partial_\theta\eta_2(\theta)=\lambda\eta_2(\theta),\;\;\theta\in(-\tau,0),\;\;
\eta_2(0)=\eta_1.
\end{equation*}
These imply that $\eta_2(\theta)=e^{\lambda\theta}\eta_1$ for each $\theta\in [-\tau,0]$ and
\begin{equation}\label{yu-9-19-4}
    [(\lambda-\kappa e^{(\gamma-\lambda)\tau})I-A_{\gamma,2}]\eta_1=0.
\end{equation}
By \eqref{yu-9-20-bbb-1}, we have
 $\zeta-\kappa e^{(\gamma-\zeta)\tau}\in \mathbb{C}_{-\frac{\beta}{2}}^+$ for all $\zeta\in \overline{\mathbb{C}_{-\alpha}^+}$,
which, combined with the second equality in (\ref{yu-1-6-20}) and $(b_4)$ above, indicates that
 the equation \eqref{yu-9-19-4} has only zero solution. Hence $\eta=0$, which is a contradiction.

\vskip 5pt
    \emph{Sub-step 1.2. We prove that for any $\mu>0$, there exist  $F_\mu\in \mathcal{L}(H_1;U)$ and $C_5(\mu)>0$  such that the solution $w_{F_\mu}(\cdot)$ of the following equation (in $H_1$):
\begin{equation}\label{yu-9-23-1}
\begin{cases}
    w_t(t)=\left\{A_{\gamma,1}+[(\rho_0+\gamma)I-A_{\gamma,1}]P^*_\beta
    (\rho_0I-\widetilde{A})^{-1}BF_\mu\right\}w(t)+\kappa e^{\gamma\tau}w(t-\tau),&t\in\mathbb{R}^+,\\
    w(0)=P^*_\beta y_0,\\
    w(t)=P^*_{\beta}\varphi(t),&t\in(-\tau,0),
\end{cases}
\end{equation}
    satisfies
\begin{equation}\label{yu-9-23-2}
    \|(w_{F_\mu}(t),w_{F_\mu}(t+\cdot))^\top\|_{\mathcal{X}}\leq C_5(\mu)e^{-\mu t}(\|y_0\|^2_{X}
    +\|\varphi\|^2_{L^2(-\tau,0;X)})^{\frac{1}{2}}\;\;\mbox{for all}\;\;t\in\mathbb{R}^+.
\end{equation}}
To this end, fix $\mu>0$ and choose $\zeta>0$ such that
\begin{equation}\label{yu-11-24-10}
\zeta>2(\mu+|\kappa| e^{(\gamma+\mu)\tau}).
\end{equation}
 On the one hand, by $(c_2)$ and \eqref{yu-11-24-5}, we have
    $\sigma(A^*_{\gamma,1})\subset \mathbb{C}_{-\beta}^+$. It then follows from \eqref{yu-12-18-1} and  \cite[Remark 5.4 in Chapter 1]{Pazy} that
    there exists a constant $C_6(\beta)>0$ such that
\begin{equation}\label{yu-1-7-1}
    \|(\lambda I-A^*_{\gamma,1})^{-1}\|_{\mathcal{L}(Q_1)}\leq C_6(\beta)\;\;\mbox{for all}\;\;\lambda\in\mathbb{C}_{-2\beta}^-.
\end{equation}
    Moreover, by $(b_1)$, $(c_1)$, \eqref{yu-9-11-1} and \eqref{yu-11-24-5}, we obtain that there exists $C_7(\beta)>0$ such
\begin{equation*}
    \|\varphi\|_{Q_1}^{2}\leq C_7(\beta)(\|(\lambda I-A^*_{\gamma,1})\varphi\|^2_{Q_1}+\|B^*P_\beta\varphi \|^2_{U})
    \;\;\mbox{for all}\;\;\lambda\in \mathbb{C}_{-4\beta}^+\;\;\mbox{and}\;\;\varphi\in Q_1.
\end{equation*}
    Combining this with \eqref{yu-1-7-1}, we deduce that
\begin{equation}\label{yu-1-7-3}
    \|\varphi\|_{Q_1}^{2}\leq C_8(\beta)(\|(\lambda I-A^*_{\gamma,1})\varphi\|^2_{Q_1}+\|B^*P_\beta\varphi \|^2_{U})
    \;\;\mbox{for all}\;\;\lambda\in \mathbb{C}\;\;\mbox{and}\;\;\varphi\in Q_1
\end{equation}
for some constant $C_8(\beta)>0$. On the other hand, by $(iii)$ in Remark \ref{yu-remark-8-00-1}, we have $B^*P_\beta= B^*(\rho_0 I-A^*)^{-1}[(\rho_0+\gamma)I-A^*_\gamma]P_\beta$.
Together with $(c_2)$, this implies $B^*P_\beta\in\mathcal{L}(Q_1;U)$. Since $Q_1$ is finite-dimensional (see \eqref{yu-12-18-1}), there exists a linear isomorphism $\Phi$ from $Q_1$ to $\mathbb{C}^N$ such that $\widehat{A}^*_{\gamma,1}:=\Phi A^*_{\gamma,1}\Phi^{-1}\in \mathbb{C}^{N\times N}$ and
    $\widehat{B}^*_1:=(B^*P_\beta)\Phi^{-1}\in\mathcal{L}(\mathbb{C}^N;U)$. Then
\eqref{yu-1-7-3} yields that there exists $C_9(\beta)>0$ such that
$\|z\|_{\mathbb{C}^N}^{2}\leq C_9(\beta) (\|(\lambda I-\widehat{A}^*_{\gamma,1})z\|^2_{\mathbb{C}^N}+\|\widehat{B}^*_1z \|^2_{U})$
    for all $\lambda\in \mathbb{C}\;\;\mbox{and}\;\;z\in \mathbb{C}^N$.
 By this inequality and  Remark \ref{yu-remark-25-9-22}, there exist $K\in\mathcal{L}(\mathbb{C}^N;U)$ and $C_{10}(\zeta,K)>0$ such that the $C_0$-semigroup $\widehat{S}_{\gamma,1,K}(\cdot)$ generated by $(\widehat{A}^*_{\gamma,1})^\top +(\widehat{B}^*_1)^*K$ satisfies
\begin{equation}\label{yu-1-7-5}
    \|\widehat{S}_{\gamma,1,K}(t)\|_{\mathcal{L}(\mathbb{C}^N)}
  \leq C_{10}(\zeta,K)e^{-\frac{\zeta}{2}t}\;\;\mbox{for all}\;\;t\in\mathbb{R}^+,
\end{equation}
   where $(\widehat{B}^*_1)^*$ denotes the adjoint operator of $\widehat{B}^*_1$ in the sense:
   $\langle \widehat{B}_1^*f,u \rangle_U=\langle f,(\widehat{B}^*_1)^*u\rangle_{\mathbb{C}^N}$ for any $f\in \mathbb{C}^N$ and $u\in U$.
\par
    Next, we show that
\begin{equation}\label{yu-1-7-6}
\begin{cases}
    (\widehat{A}^*_{\gamma,1})^\top=(\Phi^*)^{-1}A_{\gamma,1}\Phi^*,\\
    (\widehat{B}^*_1)^*=(\Phi^*)^{-1}[(\rho_0+\gamma)I-A_{\gamma,1}]P_\beta^*
    (\rho_0I-\widetilde{A})^{-1}B,
\end{cases}
\end{equation}
    where $\Phi^*\in\mathcal{L}(\mathbb{C}^N;H_1)$ is the adjoint operator of $\Phi$ in the sense:
    $\langle \Phi f,g\rangle_{\mathbb{C}^N}=\langle f,\Phi^*g\rangle_{Q_1,H_1}$ for any $f\in Q_1$ and $g\in\mathbb{C}^N$. Indeed, by the
     definitions of $\widehat{A}_{\gamma,1}^*$ and $\widehat{B}^*_1$, and $(b_1)$, we get that for any $f, g\in\mathbb{C}^N$,
\begin{eqnarray}\label{yu-1-7-7}
    \langle \widehat{A}_{\gamma,1}^*f,g\rangle_{\mathbb{C}^N}
    =\langle A_{\gamma,1}^*\Phi^{-1}f, \Phi^*g\rangle_{Q_1,H_1}
    =\langle\Phi^{-1}f, A_{\gamma,1}\Phi^*g\rangle_{Q_1,H_1}
    =\langle f,(\Phi^*)^{-1} A_{\gamma,1}\Phi^*g\rangle_{\mathbb{C}^N,\mathbb{C}^N},
\end{eqnarray}
    and that for any $z\in \mathbb{C}^N$ and $u\in U$,
\begin{eqnarray}\label{yu-1-8-1}
    \langle (B^*P_\beta)\Phi^{-1}z,u\rangle_U&=&\langle P_\beta\Phi^{-1} z,Bu\rangle_{X^*_1,X_{-1}}
    =\langle [(\rho_0+\gamma)I-A_{\gamma,1}^*]\Phi^{-1}z, P^*_\beta(\rho_0 I-\widetilde{A})^{-1}Bu\rangle_{Q_1,H_1}\nonumber\\
    &=&\langle z,(\Phi^*)^{-1}[(\rho_0+\gamma)I-A_{\gamma,1}]P_\beta^*(\rho_0I-\widetilde{A})^{-1}Bu\rangle_{\mathbb{C}^N}.
\end{eqnarray}
   Here, we have used the identity $(\Phi^*)^{-1}=(\Phi^{-1})^*$. Therefore, \eqref{yu-1-7-6} follows from \eqref{yu-1-7-7} and \eqref{yu-1-8-1}.

   Finally, by \eqref{yu-1-7-6}, we have
\begin{equation*}\label{yu-1-8-2}
    (\widehat{A}^*_{\gamma,1})^\top +(\widehat{B}^*_1)^*K=(\Phi^*)^{-1}
    \{A_{\gamma,1}+[(\rho_0+\gamma)I-A_{\gamma,1}]P_\beta^*(\rho_0I-
    \widetilde{A})^{-1}BK(\Phi^*)^{-1}\}\Phi^*.
\end{equation*}
    This shows that the operator $A_{\gamma,1}+[(\rho_0+\gamma)I-A_{\gamma,1}]P_\beta^*(\rho_0I-
    \widetilde{A})^{-1}BK(\Phi^*)^{-1}$, with domain $H_1$, generates a $C_0$-group $S_{\gamma,1,K}(\cdot)$ on $H_1$. Note that,  $S_{\gamma,1,K}(\cdot)
    =\Phi^*\widehat{S}_{\gamma,1,K}(\cdot)(\Phi^*)^{-1}$. Combining this with \eqref{yu-1-7-5},
    we obtain that there exists $C_{11}(\zeta,K)>0$ such that
\begin{equation}\label{yu-1-8-3}
     \|S_{\gamma,1,K}(t)\|_{\mathcal{L}(H_1)}
  \leq C_{11}(\zeta,K)e^{-\frac{\zeta}{2}t}\;\;\mbox{for all}\;\;t\in\mathbb{R}^+.
\end{equation}
    Since $H_1$ is finite-dimensional, $S_{\gamma,1,K}(\cdot)$ is immediately compact. Therefore, by \eqref{yu-11-24-10}, \eqref{yu-1-8-3}, and arguments analogous to those in \emph{Sub-step 1.1}, we have \eqref{yu-9-23-2} with $F_{\mu}=K(\Phi^*)^{-1}$.

\par
    \vskip 5pt
    \emph{Sub-step 1.3.  We claim that the system $[A,B]_{\kappa,\tau}$ is  rapidly stabilizable with static feedback.}
    For this purpose, fix $\mu>0$ and let $F_\mu\in \mathcal{L}(H_1;U)$ be constructed in \emph{Sub-step 1.2}. Let $y_{F_\mu}(\cdot;y_0,\varphi,\gamma)
    \in C([0,+\infty);X)$ denote  the unique solution to the following equation:
\begin{equation*}
\begin{cases}
    y_t(t)=A_\gamma y(t)+BF_\mu P^*_\beta y(t)+\kappa e^{\gamma\tau}y(t-\tau),&t\in\mathbb{R}^+,\\
    y(0)=y_0,\\
    y(t)=\varphi(t),&t\in(-\tau,0).
\end{cases}
\end{equation*}
   (Its well-posedness can be found in Remark \ref{remark-11-27-1}.) Then $y_{F_\mu}(\cdot;y_0,\varphi,\gamma)$ satisfies
\begin{equation*}\label{yu-11-24-20}
\begin{cases}
    y(t)=S_\gamma(t)y_0+\kappa e^{\gamma\tau}\int_0^tS_\gamma(t-s)y(s-\tau)\mathrm ds+
    \int_0^t\widetilde{S}_\gamma(t-s)BF_\mu P^*_\beta y(s)\mathrm ds ,&t\in [0,+\infty),\\
    y(t)=\varphi(t),&t\in(-\tau,0).
\end{cases}
\end{equation*}
     Define $Y(t):=(y_{F_\mu}(t;y_0,\varphi,\gamma), y_{F_\mu}(t+\cdot;y_0,\varphi,\gamma))^\top$ for
    $t\in\mathbb{R}^+$, and
\begin{equation*}\label{yu-11-24-22}
    \mathcal{A}_\gamma(f_1,f_2)^\top:=\left(
                                 \begin{array}{c}
                                   A_\gamma f_1 + \kappa e^{\gamma\tau}f_2(-\tau) \\
                                   \partial_\theta f_2 \\
                                 \end{array}
                               \right),
                               \;\;(f_1,f_2)^{\top}\in D(\mathcal{A}_\gamma),
\end{equation*}
 where
$D(\mathcal{A}_\gamma):=\{(f_1,f_2)^\top \in \mathcal{X}: f_1\in X_1,
    f_2\in H^1(-\tau,0;X), f_2(0)=f_1\}$.
    By Lemma~\ref{yu-lemma-4-16-1} (with $A$ and $\kappa$ replaced by $A_\gamma$ and $\kappa e^{\gamma\tau}$, respectively)
    and Proposition \ref{yu-proposition-10-14-1}, it follows that $\mathcal{A}_\gamma$ with domain $D(\mathcal{A}_\gamma)$ generates an eventually compact semigroup $\mathcal{S}_\gamma(\cdot)$ on $\mathcal{X}$, and  $Y(\cdot)$  is the unique solution to
\begin{equation}\label{yu-11-24-24}
    Y(t)=\mathcal{S}_\gamma(t)(y_0,\varphi)^\top+\int_0^t
    \widetilde{\mathcal{S}}_\gamma(t-s)
    \mathcal{B}\mathcal{F}_\mu Y(s)\mathrm ds,
\end{equation}
    where $\mathcal{B}$ is defined in \eqref{yu-4-16-12}, $\mathcal{F}_\mu:=(F_\mu P_\beta^*,0)$, and $\widetilde{\mathcal{S}}_\gamma(\cdot)$ denotes the extension of $\mathcal{S}_\gamma(\cdot)$ from $\mathcal{X}$ to $(D(\mathcal{A}_\gamma^*))'$, defined analogously to the extension of $S(\cdot)$ from $X$ to $X_{-1}$ (see (iii) in Remark \ref{yu-remark-8-00-1}). The operator $\mathcal{A}_\gamma^*$ and its domain $D(\mathcal{A}_\gamma^*)$ are defined in \eqref{yu-4-9-12} and \eqref{yu-4-12-3} (with $A^*$ and $\kappa$ replaced by $A_\gamma^*$ and $\kappa e^{\gamma\tau}$, respectively), and the norm on $D(\mathcal{A}_\gamma^*)$ is given by
   $\|(\widehat{\rho}I-\mathcal{A}_\gamma^*)W\|_{\mathcal{X}}$ for $W\in D(\mathcal{A}_\gamma^*)$
   with $\widehat{\rho}\in \rho(\mathcal{A}_\gamma^*) \cap\mathbb{R}^+$.
    By linearity of \eqref{yu-11-24-24}, we can write
\begin{equation}\label{yu-11-24-25}
    Y(t)=\Psi_{\gamma,\mu}(t)(y_0,\varphi)^\top,\;\;t\in\mathbb{R}^+.
\end{equation}

    Firstly, we prove that $\Psi_{\gamma,\mu}(\cdot)$ is a $C_0$-semigroup on $\mathcal{X}$.
     On the one hand, it is clear that
     $\Psi_{\gamma,\mu}(0)=I$. On the other hand, since $\mathcal{S}_{\gamma}(\cdot)$ is strongly continuous,
     $\Psi_{\gamma,\mu}(\cdot)$ is also strongly continuous. Hence, it suffices to show that
\begin{equation}\label{yu-1-10-1}
    \Psi_{\gamma,\mu}(t_1+t_2)=\Psi_{\gamma,\mu}(t_2)\Psi_{\gamma,\mu}(t_1)\;\;\mbox{for all}\;\;t_1,t_2\in\mathbb{R}^+.
\end{equation}
    To this end, let $Y_0:=(y_0,\varphi)^\top\in\mathcal{X}$,  and fix $t_1,t\in\mathbb{R}^+$ arbitrarily. By
    \eqref{yu-11-24-24} and the fact that $\widetilde{\mathcal{S}}_\gamma(\cdot)=\mathcal{S}_\gamma(\cdot)$ on $\mathcal{X}$, we have
\begin{eqnarray*}\label{yu-1-10-2}
    \Psi_{\gamma,\mu}(t+t_1)Y_0&=&\mathcal{S}_\gamma(t+t_1)Y_0
    +\int_0^{t+t_1}\widetilde{S}_\gamma(t+t_1-s)
    \mathcal{B}\mathcal{F}_{\mu}\Psi_{\gamma,\mu}(s)Y_0\mathrm ds \nonumber\\
    &=&\mathcal{S}_\gamma(t+t_1)Y_0+\mathcal{S}_{\gamma}(t)\int_0^{t_1}\widetilde{S}_\gamma(t_1-s)
    \mathcal{B}\mathcal{F}_{\mu}\Psi_{\gamma,\mu}(s)Y_0\mathrm ds \nonumber\\
    &\;&+\int_0^t\widetilde{S}_\gamma(t-s)
    \mathcal{B}\mathcal{F}_{\mu}\Psi_{\gamma,\mu}(s+t_1)Y_0\mathrm ds \nonumber\\
    &=&\mathcal{S}_\gamma(t)\Psi_{\gamma,\mu}(t_1)Y_0
    +\int_0^t\widetilde{S}_\gamma(t-s)
    \mathcal{B}\mathcal{F}_{\mu}\Psi_{\gamma,\mu}(s+t_1)Y_0\mathrm ds .
\end{eqnarray*}
     By the uniqueness of solution to \eqref{yu-11-24-24} (with $(y_0,\varphi)^\top$ replaced by $\Psi_{\gamma,\mu}(t_1)Y_0$) and
     the arbitrariness of $t$ and $Y_0$, it follows that
    \eqref{yu-1-10-1} holds.

  Secondly, we show that
\begin{equation}\label{yu-11-24-26}
    \int_{\mathbb{R}^+}\|\Psi_{\gamma,\mu}(t)(y_0,\varphi)^\top\|^2_{\mathcal{X}}\mathrm dt <+\infty.
\end{equation}
  To this end, let $Z_1(\cdot):=\mathcal{P}Y(\cdot)$ and $Z_2(\cdot)
   =(I-\mathcal{P})Y(\cdot)$. By the definition of $Y(\cdot)$, \eqref{yu-12-19-1} and arguments similar to those in Step 5 of Theorem 3.2 in \cite{Ma-Trelat-Wang-Yu}, we have
    $Z_1(t)=(w_{F_\mu}(t), w_{F_\mu}(t+\cdot))^\top$ for $t\in\mathbb{R}^+$,
    where $w_{F_\mu}(\cdot)$ is the solution to \eqref{yu-9-23-1}. Hence, by \eqref{yu-9-23-2},
\begin{equation}\label{yu-11-24-27}
    \|Z_1(t)\|_{\mathcal{X}}\leq C_5(\mu)e^{-\mu t}(\|y_0\|^2_{X}
    +\|\varphi\|^2_{L^2(-\tau,0;X)})^{\frac{1}{2}}\;\;\mbox{for all}\;\;t\in\mathbb{R}^+.
\end{equation}
    Let $\mathcal{S}_{\gamma,2}(\cdot)$ be as introduced in \emph{Sub-step 1.1}. It is straightforward to verify that $\mathcal{S}_{\gamma,2}(\cdot)=(I-\mathcal{P})\mathcal{S}_\gamma(\cdot)$.
   Moreover, by \eqref{yu-12-19-1}, we have $\mathcal{F}_\mu = (F_\mu,0)\mathcal{P}$. Combining this with \eqref{yu-11-24-24}, we obtain
\begin{equation}\label{yu-11-24-28}
    Z_2(t)=\mathcal{S}_{\gamma,2}(t)(I-\mathcal{P})(y_0,\varphi)^\top
    +(I-\mathcal{P})\int_0^t
    \widetilde{\mathcal{S}}_{\gamma}(t-s)
    \mathcal{B}\mathcal{F}_\mu Z_1(s)\mathrm ds .
\end{equation}
Let $T>0$ be such that
\begin{equation}\label{yu-1-8-10}
    \sqrt{2}C_{3}(\alpha)e^{-\alpha T}<1.
\end{equation}
    On the one hand, by \eqref{yu-11-24-4} and $(c_1)$,
    we have $(I-\mathcal{P}^*)D(\mathcal{A}_\gamma^*)\subset D(\mathcal{A}_\gamma^*)$. It then follows from the definition of
    $\mathcal{\widetilde{S}}_\gamma(\cdot)$ and \eqref{yu-11-24-28} that, for each $W\in D(\mathcal{A}_\gamma^*)$ and $t\in\mathbb{R}^+$,
\begin{equation}\label{wang-10-31-1}
\begin{array}{lll}
    \langle Z_2(t),W\rangle_{\mathcal{X}}
    &=&
    \langle\mathcal{S}_{\gamma,2}(t)(I-\mathcal{P})(y_0,\varphi)^\top,W\rangle_{\mathcal{X}}\\
   &&+\left\langle\int_0^t\mathcal{S}_\gamma(t-s)
(\widehat{\rho} I-\widetilde{\mathcal{A}}_\gamma)^{-1}\mathcal{B}\mathcal{F}_\mu Z_1(s)\mathrm ds ,
(\widehat{\rho} I-\mathcal{A}^*_\gamma)(I-\mathcal{P}^*)W\right\rangle_{\mathcal{X}}.
\end{array}
\end{equation}
Moreover,  by $(b_1)$ above and \eqref{yu-12-19-1}, we have $\mathcal{P}^*\mathcal{A}_\gamma^*=
    \mathcal{A}_\gamma^*\mathcal{P}^*$ on $D(\mathcal{A}^*_\gamma)$. Combining this with \eqref{wang-10-31-1} and Corollary~\ref{wang-11-8-1} (with $\mathcal{A}^*$
    replaced by $\mathcal{A}_\gamma^*$), we deduce that for each
 $W\in D(\mathcal{A}_\gamma^*)$ and each $t\in [nT,(n+1)T]$ with $n\in\mathbb{N}^+$,
\begin{eqnarray}\label{wang-10-31-2}
    &\;&\langle Z_2(t),W\rangle_X\nonumber\\
    &=&\langle \mathcal{S}_{\gamma,2}(nT)(I-\mathcal{P})(y_0,\varphi)^\top,
    \mathcal{S}_{\gamma,2}^*(t-nT)W\rangle_{\mathcal{X}}\nonumber\\
    &&+\left\langle \int_0^{t}\mathcal{S}_\gamma(t-s)
    (\widehat{\rho}I-\widetilde{\mathcal{A}}_\gamma)^{-1}
    \mathcal{B}\mathcal{F}_{\mu}Z_1(s)\mathrm ds,
    (\widehat{\rho}I-\mathcal{A}^*_\gamma)(I-\mathcal{P}^*)W\right\rangle_{\mathcal{X}}\nonumber\\
    &=&\langle \mathcal{S}_{\gamma,2}(nT)(I-\mathcal{P})(y_0,\varphi)^\top,
    \mathcal{S}_{\gamma,2}^*(t-nT)W\rangle_{\mathcal{X}}\nonumber\\
    &&+\left\langle \int_0^{t-nT}\mathcal{S}_\gamma(t-nT-s)(\widehat{\rho}I-\widetilde{\mathcal{A}}_\gamma)^{-1}
    \mathcal{B}\mathcal{F}_{\mu}Z_1(nT+s)\mathrm ds, (\widehat{\rho}I-\mathcal{A}_\gamma^*)(I-\mathcal{P}^*)W\right\rangle_{\mathcal{X}}\nonumber\\
    &&+\left\langle \int_0^{nT}\mathcal{S}_\gamma(nT-s)(\widehat{\rho}I-\widetilde{\mathcal{A}}_\gamma)^{-1}
    \mathcal{B}\mathcal{F}_{\mu}Z_1(s)\mathrm ds, (I-\mathcal{P}^*)(\widehat{\rho}I-\mathcal{A}_\gamma^*)\mathcal{S}^*_{\gamma,2}(t-nT)W
    \right\rangle_{\mathcal{X}}\nonumber\\
    &=&\langle \mathcal{S}_{\gamma,2}(t-nT)Z_2(nT),W\rangle_{\mathcal{X}}+\left\langle (I-\mathcal{P})\int_0^{t-nT}\widetilde{\mathcal{S}}_\gamma(t-nT-s)
    \mathcal{B}\mathcal{F}_\mu Z_1(nT+s)\mathrm ds,W\right\rangle_{\mathcal{X}}.
\end{eqnarray}
Since $D(\mathcal{A}^*_\gamma)$ is dense in $\mathcal{X}$, it follows from (\ref{wang-10-31-2}) that
\begin{equation}\label{yu-1-9-1}
    Z_2(t)=\mathcal{S}_{\gamma,2}(t-nT)Z_2(nT)+(I-\mathcal{P})\int_{0}^{t-nT}
    \widetilde{\mathcal{S}}_\gamma(t-nT-s)\mathcal{B}\mathcal{F}_\mu Z_1(nT+s)\mathrm ds,\;\;t\in[nT,(n+1)T].
\end{equation}
On the other hand, by Lemma~\ref{yu-lemma-10-13-2} (with $\mathcal{A}$ replaced by $\mathcal{A}_\gamma$),
there exists a constant $C_{12}(\gamma,T)>0$ such that
\begin{equation}\label{yu-1-10-11}
    \int_0^T\|\mathcal{B}^*\mathcal{S}_\gamma^*(t)W\|_{U}^2\mathrm dt
    \leq C_{12}(\gamma,T)\|W\|^2_{\mathcal{X}}\;\;\mbox{for any}\;\;W\in D(\mathcal{A}^*_\gamma).
\end{equation}
Fix any $u\in L^2(0,T;U)$. For each $W\in D(\mathcal{A}^*_\gamma)$, we have
\begin{equation*}
    \left\langle \int_0^t\widetilde{\mathcal{S}}_\gamma(t-s)\mathcal{B}u(s)\mathrm ds,W\right\rangle_{\mathcal{X}}
    =\int_0^t\langle u(s),\mathcal{B}^*\mathcal{S}_\gamma^*(t-s)W\rangle_U\mathrm ds\;\;\mbox{for all}\;\;t\in[0,T].
\end{equation*}
    Combining this with \eqref{yu-1-10-11} and applying the Cauchy-Schwarz inequality, we get
\begin{eqnarray*}\label{yu-1-10-13}
    \left|\left\langle \int_0^t\widetilde{\mathcal{S}}_\gamma(t-s)\mathcal{B}u(s)\mathrm ds,W\right\rangle_{\mathcal{X}}\right|
    &\leq& \left(\int_0^t\|u(s)\|_U^2\mathrm ds\right)^{\frac{1}{2}}
    \left(\int_0^T\|\mathcal{B}^*\mathcal{S}_\gamma^*(s)W\|_{U}^2\mathrm ds\right)^{\frac{1}{2}}\nonumber\\
    &\leq& (C_{12}(\gamma,T))^{\frac{1}{2}}\|W\|_{\mathcal{X}}\left(\int_0^t\|u(s)\|_U^2\mathrm ds\right)^{\frac{1}{2}}.
\end{eqnarray*}
Hence, for all $u\in L^2(0,T;U)$,
\begin{equation}\label{wang-11-02-1}
    \left\|\int_0^t\widetilde{\mathcal{S}}_\gamma(t-s)\mathcal{B}u(s)\mathrm ds\right\|^2_{\mathcal{X}}\leq
    C_{12}(\gamma,T)\int_0^t\|u(s)\|_U^2\mathrm ds\;\;\mbox{for all}\;\;t\in[0,T].
\end{equation}
Combining \eqref{wang-11-02-1} with \eqref{yu-9-18-1}, \eqref{yu-11-24-27}, \eqref{yu-11-24-28}, \eqref{yu-1-9-1}, and the boundedness of $\mathcal{F}_{\mu}$, we have that there exist positive constants $C_{13}(\alpha,\gamma,T,\mu,\mathcal{F}_{\mu})$ and $C_{14}(\gamma,T,\mathcal{F}_\mu)$ such that
\begin{equation}\label{yu-11-21-5}
    \sup_{t\in[0,T]}\|Z_2(t)\|_{\mathcal{X}}\leq C_{13}(\alpha,\gamma,T,\mu,\mathcal{F}_{\mu})
    \|(y_0,\varphi)^\top\|_{\mathcal{X}},
\end{equation}
and for each $n\in\mathbb{N}^+$ and $t\in [nT,(n+1)T]$,
\begin{eqnarray*}\label{yu-1-12-1}
    \|Z_2(t)\|_{\mathcal{X}}^2\leq 2(C_3(\alpha))^2e^{-2\alpha (t-nT)}\|Z_2(nT)\|^2_{\mathcal{X}}
    +C_{14}(\gamma,T,\mathcal{F}_\mu)\int_{nT}^{t}\|Z_1(s)\|^2_{\mathcal{X}}\mathrm ds.
\end{eqnarray*}
Consequently, for each $n\in\mathbb{N}^+$,
\begin{equation}\label{yu-11-21-6}
    \begin{cases}
    \sup_{t\in[nT,(n+1)T]}\|Z_2(t)\|_{\mathcal{X}}^2\leq (2(C_3(\alpha))^2+C_{14}(\gamma,T,\mathcal{F}_{\mu}))
    \left(\|Z_2(nT)\|_{\mathcal{X}}^2+\int_{nT}^{(n+1)T}\|Z_1(s)\|^2_{\mathcal{X}}\mathrm ds\right),\\
    \|Z_2((n+1)T)\|_{\mathcal{X}}^2\leq   2(C_3(\alpha))^2e^{-2\alpha T}\|Z_2(nT)\|_{\mathcal{X}}^2+C_{14}(\gamma, T,\mathcal{F}_{\mu})\int_{nT}^{(n+1)T}\|Z_1(s)\|^2_{\mathcal{X}}\mathrm ds.
\end{cases}
\end{equation}
It follows from the second inequality in \eqref{yu-11-21-6} that, for each $m\in\mathbb{N}^+$ with $m> 2$,
\begin{equation*}\label{yu-1-12-2}
    \sum_{n=2}^{m+1}\|Z_2(nT)\|_{\mathcal{X}}^2\leq 2(C_3(\alpha))^2e^{-2\alpha T}
     \sum_{n=1}^m\|Z_2(nT)\|_{\mathcal{X}}^2+C_{14}(\gamma,T,\mathcal{F}_\mu)
     \int_{T}^{(m+1)T}\|Z_1(s)\|^2_{\mathcal{X}}\mathrm ds.
\end{equation*}
Combining this with \eqref{yu-1-8-10} and \eqref{yu-11-21-5}, we obtain that for all such $m$,
\begin{eqnarray*}
    &\;&\sum_{n=1}^m\|Z_2(nT)\|_{\mathcal{X}}^2\\
    &\leq& [1-2(C_3(\alpha))^2e^{-2\alpha T}]^{-1}\left((C_{13}(\alpha,\gamma,T,\mu,\mathcal{F}_\mu))^2
    \|(y_0,\varphi)^{\top}\|_{\mathcal{X}}
    ^2+ C_{14}(\gamma,T,\mathcal{F}_\mu)\int_{0}^{(m+1)T}\|Z_1(s)\|^2_{\mathcal{X}}\mathrm ds\right).
\end{eqnarray*}
Together with \eqref{yu-11-24-27}, \eqref{yu-11-21-5}, and the first inequality in \eqref{yu-11-21-6}, this implies
 $\int_{\mathbb{R}^+}\|Z_2(t)\|_{\mathcal{X}}^2\mathrm{dt}<+\infty$.  Combining this with \eqref{yu-11-24-25} and \eqref{yu-11-24-27}, we have  \eqref{yu-11-24-26}.

Finally, by \eqref{yu-11-24-26}, \cite[Lemma 5.1, Section 5.1, Chapter 5]{Curtain-Zwart}, and the arbitrariness of $(y_0,\varphi)$,
there exist $\varepsilon>0$ and $C_{15}(\varepsilon)>0$ such that
$\|\Psi_{\gamma,\mu}(t)\|_{\mathcal{L}(\mathcal{X})}\leq C_{15}(\varepsilon)e^{-\varepsilon t}$ for all $t\in\mathbb{R}^+$.
It then follows from \eqref{yu-11-24-25} and the definition of $Y(\cdot)$ that,  for all $y_0\in X$ and $\varphi\in L^2(-\tau,0;X)$,
\begin{equation}\label{yu-11-24-31}
    \|y_{F_\mu}(t;y_0,\varphi,\gamma)\|_X\leq  C_{15}(\varepsilon)e^{-\varepsilon t}
    (\|y_0\|^2_{X}
    +\|\varphi\|^2_{L^2(-\tau,0;X)})^{\frac{1}{2}}\;\;\mbox{for all}\;\;t\in\mathbb{R}^+.
\end{equation}
   Let $K:=F_\mu P^*_\beta$, and let $y_K(\cdot;y_0,\varphi)$ be the solution to \eqref{yu-12-20-200}.
   Then $e^{\gamma t}y_K(t;y_0,\varphi)=
    y_{F_\mu}(t;y_0,e^{\gamma\cdot}\varphi,\gamma)$ for all $t\in \mathbb{R}^+$, which, together with \eqref{yu-11-24-31}, yields
\begin{equation*}\label{yu-11-24-40}
    \|y_K(t;y_0,\varphi)\|_X\leq  C_{15}(\varepsilon)e^{-\gamma t}
    (\|y_0\|^2_{X}
    +\|\varphi\|^2_{L^2(-\tau,0;X)})^{\frac{1}{2}}\;\;\mbox{for all}\;\;t\in\mathbb{R}^+.
\end{equation*}
    Since $\gamma>0$ is arbitrary, the above estimate shows that system $[A,B]_{\kappa,\tau}$ is  rapidly stabilizable with  static feedback.

\vskip 5pt
    \emph{Step 2. $(iii)\Rightarrow (ii)$.} It is obvious.

\vskip 5pt
    \emph{Step 3. $(ii)\Rightarrow (i)$.}  Firstly, we show that
\begin{equation}\label{wang-10-31-5}
      \mbox{System}\;\;\eqref{yu-10-14-2}\;\;\mbox{is rapidly stabilizable}.
\end{equation}
To this end, fix $\alpha>0$ arbitrarily. Since $[A,B]_{\kappa,\tau}$ is rapidly stabilizable, there exist $K:=K(\alpha)\in\mathcal{L}(\mathcal{X};U)$
and  $C_{16}(\alpha)>0$ such that for each $(y_0,\varphi)^\top\in \mathcal{X}$, the solution $y_K(\cdot;y_0,\varphi)$ to \eqref{yu-12-20-100} satisfies
\begin{equation}\label{yu-11-26-1}
\|y_K(t;y_0,\varphi)\|_{X}\leq C_{16}(\alpha) e^{-2\alpha t}(\|y_0\|^2_{X}
    +\|\varphi\|^2_{L^2(-\tau,0;X)})^{\frac{1}{2}}\;\;\mbox{for all}\;\;t\in\mathbb{R}^+.
\end{equation}
 This yields that, for each $t\geq \tau$,
\begin{equation*}
    \|y_K(t+\cdot;y_0,\varphi)\|^2_{L^2(-\tau,0;X)}
    =\int_{t-\tau}^t \|y_K(\sigma;y_0,\varphi)\|_X^2\mathrm d\sigma
    \leq(C_{16}(\alpha))^2\tau e^{-4\alpha (t-\tau)}(\|y_0\|^2_{X}
    +\|\varphi\|^2_{L^2(-\tau,0;X)}).
\end{equation*}
Since $y_K(\cdot;y_0,\varphi)=\varphi(\cdot)$ in $(-\tau,0)$, it follows  that there exists a constant $C_{17}(\alpha,\tau)>0$ such that
\begin{equation}\label{yu-11-26-4}
    \|y_K(t+\cdot;y_0,\varphi)\|_{L^2(-\tau,0;X)}\leq C_{17}(\alpha,\tau)
    e^{-2\alpha t}(\|y_0\|^2_{X}
    +\|\varphi\|^2_{L^2(-\tau,0;X)})^{\frac{1}{2}}\;\;\mbox{for all}\;\;
    t\in\mathbb{R}^+.
\end{equation}
     Define $u_K(t):=K(y_K(t;y_0,\varphi),
    y_K(t+\cdot;y_0,\varphi))^\top$ for $t\in\mathbb{R}^+$. Then, by \eqref{yu-11-26-1}
    and \eqref{yu-11-26-4}, we have
\begin{equation}\label{yu-11-26-5}
    e^{\alpha t}\|u_K(t)\|_U\leq C_{18}(\alpha,\tau)e^{-\alpha t}(\|y_0\|^2_{X}
    +\|\varphi\|^2_{L^2(-\tau,0;X)})^{\frac{1}{2}}\;\;\mbox{for a.e.}\;\;t\in\mathbb{R}^+
\end{equation}
  for some constant $C_{18}(\alpha,\tau)>0$. Let $Y_0:=(y_0,\varphi)^{\top}$ and define $Y_K(\cdot):=(y_1(\cdot),y_2(\cdot))^\top$ with $y_1(t):=y_K(t;y_0,\varphi)$
  and $[y_2(t)](\cdot):=y_K(t+\cdot;y_0,\varphi)$ in $(-\tau,0)$ for $t\in\mathbb{R}^+$. Then, by \eqref{yu-11-26-1},
  \eqref{yu-11-26-4} and  $(i)$ in Proposition \ref{yu-proposition-10-14-1},  $Y_K(\cdot)$ is the solution to the equation:
\begin{equation*}\label{yu-11-26-6}
    \begin{cases}
    Y_t(t)=\mathcal{A}Y(t)+\mathcal{B}u_K(t),&t\in\mathbb{R}^+,\\
    Y(0)=Y_0,
\end{cases}
\end{equation*}
  and satisfies
\begin{equation}\label{yu-11-26-7}
    \|Y_K(t)\|_{\mathcal{X}}\leq C_{19}(\alpha,\tau)e^{-2\alpha t}\|Y_0\|_{\mathcal{X}}
    \;\;\mbox{for all}\;\;t\in\mathbb{R}^+,
\end{equation}
   for some constant $C_{19}(\alpha,\tau)>0$. Set $Z_K(\cdot):=e^{\alpha\cdot}Y_K(\cdot)$ and $v_K(\cdot):=e^{\alpha\cdot}u_K(\cdot)$. Then,
  $Z_K(\cdot)$ is the solution to the equation:
\begin{equation}\label{yu-11-26-8}
\begin{cases}
    Z_t(t)=(\mathcal{A}+\alpha I)Z(t)+\mathcal{B}v_K(t), &t\in\mathbb{R}^+,\\
     Z(0)=Y_0.
\end{cases}
\end{equation}
    This, along with the definitions of $Z_K(\cdot)$ and $v_K(\cdot)$, \eqref{yu-11-26-7} and \eqref{yu-11-26-5}, implies that
    $$\inf_{u\in L^2(\mathbb{R}^+;U)}
    \int_0^{+\infty}(\|Z(t;Y_0,u)\|_{\mathcal{X}}^2+\|u(t)\|_U^2)\mathrm dt<+\infty,$$
    where $Z(\cdot;Y_0,u)$ is the solution of \eqref{yu-11-26-8}
    with $v_K$ replaced by $u$.  Since $(y_0,\varphi)^\top$ is arbitrary, it follows from \cite[Proposition 3.9]{Liu-Wang-Xu-Yu}
    that $[\mathcal{A}+\alpha I,\mathcal{B}]$ is stabilizable. The conclusion \eqref{wang-10-31-5} then follows from the arbitrariness of
    $\alpha$, \cite[Theorem 3.4]{Liu-Wang-Xu-Yu}, and standard arguments (cf. \cite[Lemma 3.1]{Ma-Trelat-Wang-Yu}).

 Secondly, we show that
\begin{equation}\label{yu-11-29-20}
    \mbox{Ker}\left(
                \begin{array}{c}
                  \lambda I-A^* \\
                  B^* \\
                \end{array}
              \right)=\{0\}\;\;\mbox{for all}\;\;\lambda\in\mathbb{C}.
\end{equation}
 Suppose, by contradiction, that there exist $\lambda_{\sharp}\in\mathbb{C}$ and  $f_{\sharp}\in X_1^*\setminus\{0\}$ such that
\begin{equation}\label{yu-11-29-21}
    (\lambda_{\sharp} I-A^*)f_{\sharp}=0\;\;\mbox{and}\;\;B^*f_{\sharp}=0.
\end{equation}
   It follows from Lemma~\ref{yu-lemma-11-29-1} that there exists $\lambda^*\in\mathbb{C}$ such that
\begin{equation}\label{yu-9-10-19}
\lambda^*-\kappa e^{-\lambda^*\tau} =\lambda_\sharp.
\end{equation}
    Define $\alpha^*:=|\mbox{Re}\lambda^*|+1$.  Then $\alpha^*>0$ and
\begin{equation}\label{yu-9-10-19-bb}
    \lambda^*\in \mathbb{C}_{-\alpha^*}^+.
\end{equation}
Let  $\varphi_{\sharp}(\theta):=\kappa e^{-\lambda^*(\theta+\tau)} f_{\sharp}$ for $\theta\in[-\tau,0]$.
    Then, by  \eqref{yu-4-12-3}, we have
    $(f_{\sharp},\varphi_{\sharp})^\top\in \mathcal{X}_1^*$  and
\begin{equation}\label{wang-10-31-6}
    (\lambda^* I-\mathcal{A}^*)(f_{\sharp},\varphi_{\sharp})^\top=(((\lambda^*-\kappa e^{-\lambda^*\tau}) I-A^*)f_{\sharp},0)^\top.
\end{equation}
 From (\ref{wang-10-31-5}) and Proposition \ref{yu-lemma-1-12-1}, there exists $C_{20}(\alpha^*)>0$ such that
\begin{equation}\label{yu-bb-11-23-100}
    \|\Psi\|^2_{\mathcal{X}}\leq C_{20}(\alpha^*)(\|(\lambda I-\mathcal{A}^*)\Psi\|^2_{\mathcal{X}}+\|\mathcal{B}^*\Psi\|_U^2)
    \;\;\mbox{for all}\;\;\lambda\in \mathbb{C}_{-\alpha^*}^+\;\;\mbox{and}\;\;\Psi\in \mathcal{X}_1^*.
\end{equation}
Applying \eqref{yu-bb-11-23-100} with $\Psi=(f_{\sharp},\varphi_{\sharp})^\top$,  and using \eqref{yu-4-16-12}, \eqref{yu-9-10-19}, \eqref{yu-9-10-19-bb}, and \eqref{wang-10-31-6}, we obtain
\begin{eqnarray*}\label{yu-8-30-4}
    \|f_{\sharp}\|_X^2&\leq&C_{20}(\alpha^*)(\|((\lambda^*-\kappa e^{-\lambda^* \tau})I-A^*)f_{\sharp}\|^2_{X}
    +\|B^*f_{\sharp}\|^2_U)\nonumber\\
    &=&C_{20}(\alpha^*)(\|(\lambda_\sharp I-A^*)f_{\sharp}\|^2_{X}
    +\|B^*f_{\sharp}\|^2_U).
\end{eqnarray*}
    Together with \eqref{yu-11-29-21}, this yields $f_{\sharp}= 0$, a contradiction. Hence \eqref{yu-11-29-20} holds.
\par
    Finally, fix $\alpha>0$ arbitrarily. By $(A_1)$ and \cite[Lemma 6.2]{Ma-Trelat-Wang-Yu},
     there exists a  projection operator  $P_{\alpha}\in \mathcal{L}(X)$ such that
     $(d_1)$ $P_{\alpha}X^*_1\subset X^*_1$ and $A^*P_{\alpha}=P_{\alpha}A^*$;
  $(d_2)$ $\sigma((P_{\alpha}A^*)|_{P_{\alpha}X})\subset \mathbb{C}_{-\alpha}^+$;
  $(d_3)$ $P_{\alpha}X$ is a finite-dimensional space;
  $(d_4)$ there exists  $C_{21}(\alpha)>0$ such that $\|(I-P_{\alpha})S^*(t)\|_{\mathcal{L}(X)}\leq C_{21}(\alpha)e^{-\frac{\alpha}{2} t}$ for all $t\in\mathbb{R}^+$.
   By $(d_1)$ and $(d_3)$, we  have
  $P_\alpha X\subset X^*_1$ and $A^*P_\alpha\in \mathcal{L}(P_\alpha X)$. It then follows from $(iii)$ in
  Remark \ref{yu-remark-8-00-1} that
    $B^*P_\alpha=B^*(\rho_0I-A^*)^{-1}(\rho_0I-A^*)P_\alpha
    =B^*(\rho_0I-A^*)^{-1}(\rho_0I-A^*P_\alpha)P_\alpha\in \mathcal{L}(P_\alpha X;U)$. Moreover, by \eqref{yu-11-29-20}, we obtain
\begin{equation}\label{yu-12-1-1}
    \mbox{Ker}\left(
                \begin{array}{c}
                  \lambda I-A^*P_\alpha \\
                  B^*P_\alpha \\
                \end{array}
              \right)=\mbox{Ker}\left(
                \begin{array}{c}
                  \lambda I-A^* \\
                  B^* \\
                \end{array}
              \right)\biggl|_{P_\alpha X}=\{0\}\;\;\mbox{for all}\;\;\lambda\in\mathbb{C}.
\end{equation}
    Let $k:=\mbox{Dim}(P_\alpha X)$ (see $(d_3)$) and let $\Phi: P_\alpha X \to \mathbb{C}^k$ be an isomorphism.  Define $\widehat{A}^*:=\Phi A^*P_\alpha\Phi^{-1}\in\mathbb{C}^{k\times k}$ and $\widehat{B}^*:=(B^*P_\alpha)\Phi^{-1}\in \mathcal{L}(\mathbb{C}^k;U)$.
    Then \eqref{yu-12-1-1} implies
\begin{equation*}\label{yu-12-1-2}
\mbox{Ker}\left(
                \begin{array}{c}
                  \lambda I-\widehat{A}^* \\
                 \widehat{B}^* \\
                \end{array}
              \right)=\{0\}\;\;\mbox{for all}\;\;\lambda\in\mathbb{C}.
\end{equation*}
    By \cite[Proposition 1.5.1, Chapter 1]{Tucsnak-Weiss},  for each $T>0$, there exists  $C_{22}(T)>0$ such that
\begin{equation}\label{yu-12-1-3}
    \|e^{\widehat{A}^* T}b\|_{\mathbb{C}^k}^2\leq
    C_{22}(T)\int_0^T\|\widehat{B}^*e^{\widehat{A}^* s}b\|_U^2\mathrm ds
    \;\;\mbox{for all}\;\;b\in\mathbb{C}^k.
\end{equation}
    (Here, it should be emphasized that the proof of \cite[Proposition 1.5.1,  Chapter 1]{Tucsnak-Weiss} is independent of the dimension of control space.)
    Since $P_\alpha X\subset X_1^*$, by the definitions of $\widehat{A}^*$ and $\widehat{B}^*$, and \eqref{yu-12-1-3},
    we have that there exists a constant $C_{23}(T)>0$
\begin{equation}\label{yu-12-1-4}
    \|P_\alpha S^*(T)\varphi\|_{X}^2\leq
    C_{23}(T)\int_0^T\|B^*P_\alpha S^*(s)\varphi\|_U^2\mathrm ds
    \;\;\mbox{for all}\;\;\varphi\in X.
\end{equation}
    Fix $T_0>0$ and let $T>T_0>0$ with $N:=[T/T_0]\geq2$, where $[T/T_0]:=\max\{n\in\mathbb{N}:nT_0\leq T\}$.
    For  $\varphi\in X_1^*$, using \eqref{yu-12-1-4} and the semigroup property, we obtain
\begin{eqnarray*}
    \|S^*(T)\varphi\|_X^2&=&\|S^*(T-NT_0)S^*(NT_0)\varphi\|_X^2\\
    &\leq&2D(T_0)
    (\|P_\alpha S^*(NT_0)\varphi\|_X^2+\|(I-P_\alpha)S^*(NT_0)\varphi\|_X^2)\nonumber\\
    &\leq&2D(T_0)
    \left(C_{23}(T_0)\int_0^{T_0}\|B^*P_\alpha S^*(t+(N-1)T_0)\varphi\|_U^2\mathrm dt
    +\|(I-P_\alpha)S^*(NT_0)\varphi\|_X^2\right),
\end{eqnarray*}
    where $D(T_0):=\sup_{t\in[0,T_0]}\|S^*(t)\|^2_{\mathcal{L}(X)}$.
    This, along with $(d_4)$, implies that
\begin{equation}\label{yu-12-1-11}
    \|S^*(T)\varphi\|_X^2\leq I_1+I_2,
\end{equation}
    where
\begin{equation}\label{yu-12-1-12}
\begin{cases}
    I_1:=2D(T_0)\left[2C_{23}(T_0)\int_{(N-1)T_0}^{NT_0}\|B^*S^*(t)\varphi\|_U^2\mathrm dt
    +(C_{21}(\alpha))^2e^{-NT_0\alpha}\|\varphi\|_X^2\right],\\
    I_2:=4D(T_0)C_{23}(T_0)\int_{0}^{T_0}\|B^*(I-P_\alpha)S^*(t)S^*((N-1)T_0)\varphi\|_U^2
    \mathrm dt.
\end{cases}
\end{equation}
   By Assumption $(A_3)$, $(iv)$ of Remark~\ref{yu-remark-8-00-1} and $(d_4)$, there exists  a constant $C_{24}(T_0,\alpha)>0$ such that
\begin{equation*}
    I_2\leq C_{24}(T_0,\alpha)e^{-(N-1)T_0\alpha}\|\varphi\|_X^2.
\end{equation*}
    Combining the above estimate with \eqref{yu-12-1-11}, \eqref{yu-12-1-12} and using the fact $NT_0\leq T<(N+1)T_0$, we have
\begin{equation*}\label{yu-12-1-14}
    \|S^*(T)\varphi\|_X^2
    \leq 4D(T_0)C_{23}(T_0)\int_0^T\|B^*S^*(t)\varphi\|_U^2\mathrm dt
    +e^{T_0\alpha}[2D(T_0)(C_{21}(\alpha))^2+C_{24}(T_0,\alpha)e^{\alpha T_0}]e^{-\alpha T}\|\varphi\|_X^2.
\end{equation*}
    Since $\varphi$, $\alpha$ and $T$ are arbitrary,  the conclusion follows from \cite[Theorem 3.4]{Liu-Wang-Xu-Yu}, namely that system
     $[A,B]$ is rapidly stabilizable.
\par
    In summary, we complete the proof of Theorem \ref{yu-theorem-9-01-1}.
\end{proof}

\section{Applications}\label{yu-sec-10-1-4}
In this section, we provide two illustrative examples as applications of Theorem \ref{yu-theorem-9-01-1}.

\subsection{Delayed heat equation with Neumann boundary control}  Let $\alpha>0$, $\tau>0$ and $\kappa\in\mathbb{R}$.
     Consider the following controlled equation:
\begin{equation}\label{m-25-7-21-1}
\begin{cases}
    y_t(t,x)=(\partial_x^2-\alpha ) y(t,x)+\kappa y(t-\tau,x),&(t,x)\in \mathbb{R}^+\times (0,1),\\
    y_x(t,0)=0,\;\;y_x(t,1)=u(t),&t\in \mathbb{R}^+,\\
    y(0,x)=y_0(x),&x\in(0,1),\\
    y(t,x)=\varphi(t,x),&(t,x)\in(-\tau,0)\times(0,1),
\end{cases}
\end{equation}
    where  $u\in L^2_{loc}(\mathbb{R}^+;\mathbb{C})$,  $y_0\in L^2(0,1;\mathbb{C})$ and $\varphi\in L^2(-\tau,0;L^2(0,1;\mathbb{C}))$.
\begin{remark}\label{m-25-7-21-2}
    Two remarks on system \eqref{m-25-7-21-1} are in order.
\begin{enumerate}
  \item [$(i)$] When $\kappa=0$, system \eqref{m-25-7-21-1} is null controllable in $L^2(0,1;\mathbb{C})$ (see \cite[Theorem 3.3]{Fattorini-Russell}).
  \item [$(ii)$] When $\kappa\neq 0$, system \eqref{m-25-7-21-1} is not null controllable in $L^2(0,1;\mathbb{C})$.
  This follows from   \cite[Theorem 2.1]{Wang-Yu-Zhang},  or can alternatively be proved by using the approach in \cite[Section 5.2]{Khodja-Bouzidi-Dupaix-Maniar}.
\end{enumerate}
\end{remark}
\begin{theorem}
    The following statements hold:
\begin{enumerate}
\item [$(i)$] For each $(y_0,\varphi)^\top\in L^2(0,1;\mathbb{C})\times L^2(-\tau,0;L^2(0,1;\mathbb{C}))$ and $u\in L^2_{loc}(\mathbb{R}^+;\mathbb{C})$,
system \eqref{m-25-7-21-1} admits a unique solution in $C([0,+\infty);L^2(0,1;\mathbb{C}))$.
  \item [$(ii)$] For each  $f\in L^2(0,1;\mathbb{C})$ and $(y_0,\varphi)^\top\in L^2(0,1;\mathbb{C})\times L^2(-\tau,0;L^2(0,1;\mathbb{C}))$,  the closed-loop system
\begin{equation}\label{yu-9-22-1}
\begin{cases}
     y_t(t,x)=(\partial_x^2-\alpha ) y(t,x)+\kappa y(t-\tau,x),&(t,x)\in \mathbb{R}^+\times (0,1),\\
    y_x(t,0)=0,\;\;y_x(t,1)=\langle f, y(t)\rangle_{L^2(0,1;\mathbb{C})},&t\in \mathbb{R}^+,\\
    y(0,x)=y_0(x),&x\in(0,1),\\
    y(t,x)=\varphi(t,x),&(t,x)\in(-\tau,0)\times(0,1),
\end{cases}
\end{equation}
   admits a unique solution $y_f(\cdot;y_0,\varphi)\in C([0,+\infty);L^2(0,1;\mathbb{C}))$.
  \item [$(iii)$] For each $\gamma>0$, there exists a feedback function $f\in L^2(0,1;\mathbb{C})$ and a constant $C(\gamma)>0$ such that, for every $(y_0,\varphi)^\top\in L^2(0,1;\mathbb{C})\times L^2(-\tau,0;L^2(0,1;\mathbb{C}))$, the corresponding solution $y_f(\cdot;y_0,\varphi)$ to system \eqref{yu-9-22-1} satisfies
  \begin{equation*}\label{yu-9-22-2}
    \|y_f(t;y_0,\varphi)\|_{L^2(0,1;\mathbb{C})}\leq C(\gamma)e^{-\gamma t}
    (\|y_0\|_{L^2(0,1;\mathbb{C})}^2+\|\varphi\|^2_{L^2(-\tau,0;L^2(0,1;\mathbb{C}))})^{\frac{1}{2}}
    \;\;\mbox{for all}\;\;t\in\mathbb{R}^+.
  \end{equation*}
\end{enumerate}
\end{theorem}
\begin{proof}
    Let $X=L^2(0,1;\mathbb{C})$ and $U=\mathbb{C}$. Define
    $A:=\partial_x^2-\alpha$ with domain $D(A):=\{f\in H^2(0,1;\mathbb{C}):f_x(0)=f_x(1)=0\}$. It is well known that $A$  generates an analytic and immediately compact semigroup on $X$. Thus,  $A$ satisfies Assumption $(A_1)$.
    By the same arguments as  in $(iii)$ of Remark \ref{yu-remark-8-00-1}, the operator $A$ admits a unique extension to $D(A)'$ with domain
    $X$, where $D(A)'$ is the dual space of $D(A)$ with respect to the pivot space $X$. (Here, we note that $A=A^*$ in our setting.)  This extension is denoted by $\widetilde{A}$.
    Moreover, since $\alpha>0$, for each $f\in \mathbb{C}$, the boundary value problem:
\begin{equation*}
\left\{
\begin{array}{l}
   (\partial_x^2-\alpha)\varphi=0\;\;\mbox{in}\;\;(0,1),\\
    \varphi_x(0)=0,\;\;\varphi_x(1)=f,
\end{array}\right.
\end{equation*}
admits a unique solution $\varphi_f$ in $ H^2(0,1;\mathbb{C})$. We define the Neumann map $\mathcal{N}: U\to
X$ by $\mathcal{N}f=\varphi_f$, and set $B:=-\widetilde{A}\mathcal{N}$.
    By \cite[Section 3.3, Chapter 3]{Lasiecka-Triggiani-2000}, we have  $\mathcal{N}\in \mathcal{L}(\mathbb{C};H^{\frac{3}{2}}(0,1;\mathbb{C}))$, and
\begin{equation}\label{yu-25-9-24-3}
    B\in \mathcal{L}(\mathbb{C};(D((-A)^{\frac{1}{4}+\varepsilon}))')\;\;\mbox{for all}\;\;\varepsilon>0,
\end{equation}
    where $(D((-A)^{\frac{1}{4}+\varepsilon}))'$ denotes the dual space of $D((-A)^{\frac{1}{4}+\varepsilon})$
    with respect to the pivot space $X$.
    By
    \eqref{yu-25-9-24-3},  $B$ satisfies Assumption $(A_2)$.
    Since $A$ generates an analytic semigroup $S(\cdot)$ on $X$, it follows from \eqref{yu-25-9-24-3} that
     \eqref{yu-5-4-2} holds for each $T>0$. Therefore, system $[A,B]$ satisfies Assumption $(A_3)$.
    Consequently, assertions $(i)$ and $(ii)$ follow from $(v)$ in Remark \ref{yu-remark-8-00-1} and Remark \ref{remark-11-27-1}, respectively.
   Finally, by $(i)$ in Remark \ref{m-25-7-21-2} and \cite[Theorem 3.4]{Liu-Wang-Xu-Yu}, system \eqref{m-25-7-21-1} with $\kappa=0$ is rapidly stabilizable in $L^2(0,1;\mathbb{C})$. This, together with Theorem \ref{yu-theorem-9-01-1}, implies $(iii)$.
    The proof is complete.
\end{proof}

\subsection{Delayed heat equation with internal control}\label{yu-sec-10-1}
    Let $(a,\tau)\in\mathbb{R}\times\mathbb{R}^+$, and let $\Omega$ be a bounded, connected, open subset in $\mathbb{R}^d$ with smooth boundary $\partial\Omega$. Let $\omega_1$ and $\omega_2$ be two subsets with nonempty interior in $\Omega$.  Consider the following controlled equation:
\begin{equation}\label{yu-25-9-24-4}
\begin{cases}
    y_t(t,x)=\triangle y(t,x)+a\chi_{\omega_1}y(t-\tau,x)+\chi_{\omega_2}u(t,x),&(t,x)\in\mathbb{R}^+\times\Omega,\\
    y(t,x)=0,&(t,x)\in \mathbb{R}^+\times\partial\Omega,\\
    y(0,x)=y_0(x),&x\in \Omega,\\
   y(t,x)=\varphi(t,x),&(t,x)\in(-\tau,0)\times\Omega,
\end{cases}
\end{equation}
    where $y_0\in L^2(\Omega;\mathbb{C})$, $\varphi\in L^2(-\tau,0;L^2(\Omega;\mathbb{C}))$ and $u\in L^2_{loc}(\mathbb{R}^+;L^2(\Omega;\mathbb{C}))$.
\begin{remark}\label{yu-remark-25-9-24-1}
    Several remarks  concerning system \eqref{yu-25-9-24-4} are in order.
\begin{enumerate}
  \item [$(i)$] For each $(y_0,\varphi)\in L^2(\Omega;\mathbb{C})\times L^2(-\tau,0;L^2(\Omega;\mathbb{C}))$ and $u\in L^2_{loc}(\mathbb{R}^+;L^2(\Omega;\mathbb{C}))$, system \eqref{yu-25-9-24-4} has a unique solution in
      $C([0,+\infty);L^2(\Omega;\mathbb{C}))$.
  \item [$(ii)$] When $a=0$, system \eqref{yu-25-9-24-4} is null controllable in $L^2(\Omega;\mathbb{C})$ (see \cite[Section 5.2.5.1, Chapter 5]{Trelat}). Consequently, by \cite[Theorem 1.1]{Liu-Wang-Xu-Yu}, it is rapidly stabilizable in $L^2(\Omega;\mathbb{C})$.
  \item [$(iii)$]  When $a\neq 0$, system \eqref{yu-25-9-24-4} is not null controllable in $L^2(\Omega;\mathbb{C})$. This follows from \cite[Theorem 2.1]{Wang-Yu-Zhang} or can be proved by using the approach in \cite[Section 5.2]{Khodja-Bouzidi-Dupaix-Maniar}.
\end{enumerate}
\end{remark}
\begin{theorem}\label{yu-theorem-10-1-1}
    Suppose that $\omega_1 \cup \omega_2 = \Omega$.  Then, for each $\gamma>0$, there exists an operator
$F_\gamma \in \mathcal{L}(L^2(\Omega;\mathbb{C}))$ such that, defining $K_\gamma: L^2(\Omega;\mathbb{C}) \times C([-\tau,0];L^2(\Omega;\mathbb{C})) \to L^2(\Omega;\mathbb{C})$ by
\begin{equation}\label{yu-25-9-24-9-bbb}
    K_\gamma(h,\varphi)^\top:=F_\gamma h+a\chi_{\Omega\setminus\omega_1}\varphi(-\tau),\;\;(h,\varphi)^\top\in
    L^2(\Omega;\mathbb{C})\times C([-\tau,0];L^2(\Omega;\mathbb{C})),
\end{equation}
     and setting
\begin{equation}\label{yu-25-9-24-10}
    L_\gamma(t):=
\begin{cases}
    0&\mbox{if}\;\;t\in[0,\tau],\\
    K_\gamma&\mbox{if}\;\;t\in(\tau,+\infty),
\end{cases}
\end{equation}
     we have
\begin{equation}\label{yu-26-4-16-1}
    K_\gamma\in\mathcal{L}(L^2(\Omega;\mathbb{C})\times C([-\tau,0];L^2(\Omega;\mathbb{C}));L^2(\Omega;\mathbb{C})),
\end{equation}
  and the following statements hold:
    \begin{enumerate}
      \item [$(i)$] For each $(y_0,\varphi)^\top\in L^2(\Omega;\mathbb{C})\times L^2(-\tau,0;L^2(\Omega;\mathbb{C}))$,
      the closed-loop system
\begin{equation}\label{yu-25-9-24-50}
\begin{cases}
    y_t(t,x)=\triangle y(t,x)+a\chi_{\omega_1}y(t-\tau,x)\\
    \;\;\;\;\;\;\;\;\;\;\;\;\;\;\;\;+\chi_{\omega_2}L_\gamma(t)(y(t,\cdot),y(t+\cdot,\cdot))^\top(x),
    &(t,x)\in\mathbb{R}^+\times\Omega,\\
    y(t,x)=0,&(t,x)\in \mathbb{R}^+\times\partial\Omega,\\
    y(0,x)=y_0(x),&x\in \Omega,\\
    y(t,x)=\varphi(t,x),&(t,x)\in(-\tau,0)\times\Omega,
\end{cases}
\end{equation}
     admits a unique mild solution $y_\gamma(\cdot;y_0,\varphi)\in C([0,+\infty);L^2(\Omega;\mathbb{C}))$.
        \item [$(ii)$] For each $(y_0,\varphi)^\top\in L^2(\Omega;\mathbb{C})\times L^2(-\tau,0;L^2(\Omega;\mathbb{C}))$, the function $f: \mathbb{R}^+\rightarrow L^2(\Omega;\mathbb{C})$, defined by
        $$
        f(t):=L_\gamma(t)(y_\gamma(t;y_0,\varphi),y_\gamma(t+\cdot;y_0,\varphi))^\top\;\;\mbox{for a.e.}\;\;t\in \mathbb{R}^+,
        $$
             belongs to $L^2(\mathbb{R}^+;L^2(\Omega;\mathbb{C}))$.
      \item [$(iii)$] There exists a constant $C(\gamma)>0$ such that, for any $(y_0,\varphi)^\top\in L^2(\Omega;\mathbb{C})\times L^2(-\tau,0;L^2(\Omega;\mathbb{C}))$,
\begin{equation}\label{yu-25-9-24-51}
     \|y_\gamma(t;y_0,\varphi)\|_{L^2(\Omega;\mathbb{C})}\leq C(\gamma)e^{-\gamma t}
    (\|y_0\|_{L^2(\Omega;\mathbb{C})}^2+\|\varphi\|^2_{L^2(-\tau,0;L^2(\Omega;\mathbb{C}))})^{\frac{1}{2}}
    \;\;\mbox{for all}\;\;t\in\mathbb{R}^+.
\end{equation}
    \end{enumerate}
\end{theorem}

\begin{proof}
    Let $X:=L^2(\Omega;\mathbb{C})$ and $U:=L^2(\Omega;\mathbb{C})$. Define $A:=\triangle$ with domain $D(A):=H_0^1(\Omega;\mathbb{C})\cap H^2(\Omega;\mathbb{C})$, and let $B:=\chi_{\omega_2}$.
    Then system $[A,B]$ satisfies Assumptions $(A_1)$, $(A_2)$ and $(A_3)$. Combining this with $(ii)$ in Remark \ref{yu-remark-25-9-24-1} and Theorem \ref{yu-theorem-9-01-1}, we deduce that for each $\gamma>0$, there exist $F_\gamma\in \mathcal{L}(X;U)$ and $C_1(\gamma)>0$ such that, for any $(y_0,\varphi)\in X\times L^2(-\tau,0;X)$, the solution $z_{F_\gamma}(\cdot;y_0,\varphi)$ to
\begin{equation}\label{yu-25-9-24-7}
\begin{cases}
    z_t(t,x)=\triangle z(t,x)+az(t-\tau,x)+\chi_{\omega_2}(x)F_\gamma [z(t,\cdot)](x),&(t,x)\in\mathbb{R}^+\times\Omega,\\
    z(t,x)=0,&(t,x)\in \mathbb{R}^+\times\partial\Omega,\\
    z(0,x)=y_0(x),&x\in \Omega,\\
    z(t,x)=\varphi(t,x),&(t,x)\in(-\tau,0)\times\Omega,
\end{cases}
\end{equation}
    satisfies
\begin{equation}\label{yu-25-9-24-8}
    \|z_{F_\gamma}(t;y_0,\varphi)\|_{X}\leq C_1(\gamma)e^{-\gamma t}
    (\|y_0\|_{X}^2+\|\varphi\|^2_{L^2(-\tau,0;X)})^{\frac{1}{2}}
    \;\;\mbox{for all}\;\;t\in\mathbb{R}^+.
\end{equation}
Here, we note that the equation \eqref{yu-25-9-24-7} has a unique solution in $C([0,+\infty);X)$ (see Remark~\ref{remark-11-27-1} below).
    Fix $\gamma>0$.  Let  $K_\gamma: X\times C([-\tau,0];X)\to U$ be defined by \eqref{yu-25-9-24-9-bbb}.
  Then, for any $(h,\varphi)^\top\in X\times C([-\tau,0];X)$,
\begin{equation}\label{yu-25-9-24-9-b}
    \|K_\gamma(h,\varphi)^\top\|_U\leq (\|F_\gamma\|_{\mathcal{L}(X;U)}+|a|)
    (\|h\|_X+\max_{t\in[-\tau,0]}\|\varphi(t)\|_X),
\end{equation}
   which implies \eqref{yu-26-4-16-1}.

\par
    In the rest of the proof, we show $(i)$-$(iii)$ one by one. Before doing it, we first observe
     that, since $\omega_1\cup \omega_2=\Omega$, $\chi_{\omega_2}\chi_{\Omega\setminus\omega_1}=\chi_{\Omega\setminus\omega_1}$ and the  following claim holds: \textbf{(CL)} \emph{for any $(y_0,\varphi)^\top\in X\times C([-\tau,0];X)$,
    $z_{F_\gamma}(\cdot;y_0,\varphi)$ is the solution to \eqref{yu-25-9-24-7} if and only if it is the solution to the following equation:}
\begin{equation*}\label{yu-25-9-24-9}
\begin{cases}
    y_t(t,x)=\triangle y(t,x)+a\chi_{\omega_1}y(t-\tau,x)+\chi_{\omega_2}K_\gamma(y(t,\cdot),y(t+\cdot,\cdot))^\top(x),
    &(t,x)\in\mathbb{R}^+\times\Omega,\\
    y(t,x)=0,&(t,x)\in \mathbb{R}^+\times\partial\Omega,\\
    y(0,x)=y_0(x),&x\in \Omega,\\
    y(t,x)=\varphi(t,x),&(t,x)\in(-\tau,0)\times\Omega.
\end{cases}
\end{equation*}

\vskip 5pt
\emph{Proof of $(i)$.}  Let $(y_0,\varphi)\in X\times L^2(-\tau,0;X)$. By \eqref{yu-25-9-24-10}, the system \eqref{yu-25-9-24-50} on $[0,\tau]$ reduces to
\begin{equation}\label{yu-25-9-25-1}
\begin{cases}
    p_t(t,x)=\triangle p(t,x)+a\chi_{\omega_1}p(t-\tau,x),
    &(t,x)\in(0,\tau)\times\Omega,\\
    p(t,x)=0,&(t,x)\in (0,\tau)\times\partial\Omega,\\
    p(0,x)=y_0(x),&x\in \Omega,\\
   p(t,x)=\varphi(t,x),&(t,x)\in(-\tau,0)\times\Omega.
\end{cases}
\end{equation}
    By standard energy estimates, \eqref{yu-25-9-25-1} admits a unique solution $p\in C([0,\tau];X)$ and there exists  $C_2(\tau)>0$ such that
\begin{equation}\label{yu-25-9-25-2}
    \max_{t\in[0,\tau]}\|p(t)\|_X\leq C_2(\tau)
    (\|y_0\|_{X}^2+\|\varphi\|^2_{L^2(-\tau,0;X)})^{\frac{1}{2}}.
\end{equation}
   Next, let  $q$ be the unique solution to the following equation:
\begin{equation}\label{yu-25-9-25-3}
\begin{cases}
    q_t(t,x)=\triangle q(t,x)+a\chi_{\omega_1}q(t-\tau,x)+\chi_{\omega_2}K_\gamma[(q(t,\cdot),q(t+\cdot,\cdot))^\top](x),
    &(t,x)\in\mathbb{R}^+\times\Omega,\\
    q(t,x)=0,&(t,x)\in \mathbb{R}^+\times\partial\Omega,\\
    q(0,x)=p(\tau,x),&x\in \Omega,\\
   q(t,x)=p(t+\tau,x),&(t,x)\in(-\tau,0)\times\Omega.
\end{cases}
\end{equation}
   By \textbf{(CL)} and \eqref{yu-25-9-24-8},
\begin{equation}\label{yu-25-9-25-4}
\begin{array}{lll}
    \|q(t)\|_X&\leq& C_1(\gamma)e^{-\gamma t} (\|p(\tau)\|_{X}^2+\|p(\cdot+\tau)\|^2_{L^2(-\tau,0;X)})^{\frac{1}{2}}\\
    &\leq&\sqrt{1+\tau} C_1(\gamma)e^{-\gamma t}\max_{s\in[0,\tau]}\|p(s)\|_X\;\;\mbox{for all}\;\;t\in\mathbb{R}^+.
\end{array}
\end{equation}
    Define
\begin{equation}\label{yu-25-9-25-5-b}
 y_\gamma(t;y_0,\varphi):=
\begin{cases}
    p(t)&\mbox{if}\;\;t\in (-\tau,\tau],\\
    q(t-\tau)&\mbox{if}\;\;t\in(\tau,+\infty).
\end{cases}
\end{equation}
    Then, by \eqref{yu-25-9-24-10}, \eqref{yu-25-9-25-1} and \eqref{yu-25-9-25-3},  $y_\gamma(\cdot;y_0,\varphi)$ belongs to $C([0,+\infty);X)$ and solves the equation
    \eqref{yu-25-9-24-50}. Thus, system \eqref{yu-25-9-24-50} admits a solution. Uniqueness follows by the same decomposition argument. Hence
     $(i)$ holds and the solution of system \eqref{yu-25-9-24-50} is given by \eqref{yu-25-9-25-5-b}.

\vskip 5pt
\emph{Proof of $(ii)$.} By \eqref{yu-25-9-25-2}, \eqref{yu-25-9-25-4} and \eqref{yu-25-9-25-5-b}, we have
\begin{equation}\label{yu-25-9-25-6}
    \|y_\gamma(t;y_0,\varphi)\|_X\leq \sqrt{1+\tau}C_1(\gamma)C_2(\tau)e^{\gamma\tau} e^{-\gamma t}(\|y_0\|_{X}^2+\|\varphi\|^2_{L^2(-\tau,0;X)})^{\frac{1}{2}}\;\;\mbox{for all}\;\;t\geq \tau.
\end{equation}
 This, along with \eqref{yu-25-9-24-9-b} and \eqref{yu-25-9-24-10}, implies that
\begin{eqnarray*}\label{yu-25-9-25-7}
    &\;&\|f(t)\|_U\nonumber\\&=&\|L_\gamma(t)(y_\gamma(t;y_0,\varphi), y_\gamma(t+\cdot;y_0,\varphi))^\top\|_U\nonumber\\
    &\leq&
\begin{cases}
    0&\mbox{if}\;\;t\in[0,\tau],\\
    2(\|F_\gamma\|_{\mathcal{L}(X;U)}+|a|)\max_{s\in[0,2\tau]}\| y_\gamma(s;y_0,\varphi)\|_X
    &\mbox{if}\;\;t\in(\tau,2\tau],\\
    \sqrt{1+\tau}(\|F_\gamma\|_{\mathcal{L}(X;U)}+|a|)C_1(\gamma)C_2(\tau)(1+e^{\gamma\tau})e^{\gamma\tau} e^{-\gamma t}
    (\|y_0\|_{X}^2+\|\varphi\|^2_{L^2(-\tau,0;X)})^{\frac{1}{2}}&\mbox{if}\;\;t>2\tau.
\end{cases}
\end{eqnarray*}
   Thus, $(ii)$ holds.
\vskip 5pt
   \emph{Proof of $(iii)$.} From \eqref{yu-25-9-25-2} and \eqref{yu-25-9-25-5-b}, we have that for $t\in[0,\tau]$,
\begin{equation*}\label{yu-25-9-25-8}
     \|y_\gamma(t;y_0,\varphi)\|_X \leq C_2(\tau)e^{\gamma\tau}e^{-\gamma t}(\|y_0\|_{X}^2+\|\varphi\|^2_{L^2(-\tau,0;X)})^{\frac{1}{2}}
     \;\;\mbox{for each}\;\;t\in[0, \tau].
\end{equation*}
    Together with \eqref{yu-25-9-25-6}, this yields \eqref{yu-25-9-24-51} with
    $C(\gamma):=\sqrt{1+\tau}(C_1(\gamma)+1)C_2(\tau)e^{\gamma\tau}$. Hence, $(iii)$ holds.
\par
    This completes the proof.
\end{proof}
\begin{remark}
  Theorem \ref{yu-theorem-10-1-1} not only shows that system \eqref{yu-25-9-24-4} is rapidly stabilizable, but, through its proof combined with that of Theorem \ref{yu-theorem-9-01-1}, it also provides an effective design scheme for the feedback law.
    In particular, by \eqref{yu-25-9-24-9-bbb} and \eqref{yu-25-9-24-10}, the feedback control constructed in  Theorem \ref{yu-theorem-10-1-1} is given by
\begin{equation}\label{yu-6-25-1}
    L_\gamma(t)(y(t),y(t+\cdot))=
    \begin{cases}
        0&\mbox{if}\;\;t\in(0,\tau),\\
        F_\gamma y(t)+a\chi_{\Omega\setminus\omega_1}y(t-\tau)&\mbox{if}\;\;t\in(\tau,+\infty),
    \end{cases}
\end{equation}
    where $y(\cdot)$ denotes the state trajectory. Therefore, the feedback control at time $t$ depends only on the current state $y(t)$ and the delayed state $y(t-\tau)$.
This reveals a remarkably simple structure of the stabilizing feedback law, which can be interpreted as a low-complexity dynamic feedback mechanism.
In contrast to classical approaches based on the state-space approach, the present design is explicitly implementable.

\end{remark}
\section{Appendix}\label{yu-sec-6.2}

In this section, let $\tau>0$, $E\in \mathcal{L}(X)$ and $F\in\mathcal{L}(X\times L^2(-\tau,0;X);U)$ be fixed. We consider the well-posedness of the following equation:
\begin{equation}\label{yu-11-26-10}
\begin{cases}
    y_t(t)=Ay(t)+Ey(t-\tau)+B[F(y(t),
    y(t+\cdot))^\top],&t\in\mathbb{R}^+,\\
    y(0)=y_0,\\
    y(t)=\varphi(t),&t\in(-\tau,0),
\end{cases}
\end{equation}
    where $y_0\in X$ and $\varphi\in L^2(-\tau,0;X)$.
\begin{proposition}\label{yu-proposition-11-26-2}
    Suppose that  assumptions $(A_1)$-$(A_3)$ hold. Then for each $y_0\in X$ and $\varphi\in L^2(-\tau,0;X)$, \eqref{yu-11-26-10} has a unique solution in $C([0,+\infty);X)$.
\end{proposition}

\begin{proof}
    Let $y_0\in X$ and $\varphi\in L^2(-\tau,0;X)$ be fixed arbitrarily.
    It suffices to show that the following equation:
\begin{equation}\label{yu-11-26-11}
\begin{cases}
    y(t)=S(t)y_0+\int_0^tS(t-s)Ey(s-\tau)\mathrm ds+\int_0^t\widetilde{S}(t-s)B[F(y(s),
    y(s+\cdot))^\top]\mathrm ds,
    &t\in [0,+\infty),\\
    y(t)=\varphi(t),&t\in(-\tau,0)
\end{cases}
\end{equation}
    has a unique solution in $C([0,+\infty);X)$.
     We will use the Banach fixed-point theorem to prove it.

Indeed, according to the same arguments as those led to \eqref{wang-11-02-1},  Assumption $(A_3)$ and $(iv)$ in Remark \ref{yu-remark-8-00-1},
there exists a constant $C(\tau)>0$ so that for any $u\in L^2(0,\tau;U)$,
\begin{equation}\label{yu-11-26-12}
    \left\|\int_0^t\widetilde{S}(t-s)Bu(s)\mathrm ds\right\|^2_X\leq
    C(\tau)\int_0^t\|u(s)\|_U^2\mathrm ds\;\;\mbox{for all}\;\;t\in[0,\tau].
\end{equation}
Let $T\in(0,\tau]$ be small enough so that
\begin{equation}\label{yu-11-26-15}
    \sqrt{(1+\tau)TC(\tau)}\|F\|_{\mathcal{L}(X\times L^2(-\tau,0;X);U)}<1\;\;\mbox{and}\;\;n^*:=\tau/T\in\mathbb{N}^+.
\end{equation}
    Define $\mathcal{E}_{T,y_0,\varphi}:=\{f:(-\tau,T]\to X\,|\,
    f|_{[0,T]}\in C([0,T];X),\;f(0)=y_0,\;f|_{(-\tau,0)}=\varphi\}$.
    It is clear that $\mathcal{E}_{T,y_0,\varphi}$ is a complete metric space with the metric  $d(f,g):=\max_{t\in[0,T]}\|f(t)-g(t)\|_X$
    for any $f,g\in \mathcal{E}_{T,y_0,\varphi}$. Define an operator $\mathcal{G}$ on $\mathcal{E}_{T,y_0,\varphi}$ as follows:
\begin{equation*}\label{yu-11-26-16}
    \mathcal{G}[y](t):=
\begin{cases}
    S(t)y_0+\int_0^tS(t-s)E\varphi(s-\tau)\mathrm ds
    +\int_0^t\widetilde{S}(t-s)B[F(y(s),
    y(s+\cdot))^\top]\mathrm ds,
    &t\in[0,T],\\
    \varphi(t),&t\in(-\tau,0).
\end{cases}
\end{equation*}
    Note that for each $y\in\mathcal{E}_{T,y_0,\varphi}$,
\begin{equation}\label{yu-11-27-2}
\begin{array}{lll}
    \int_0^T\left(\|y(t)\|_X^2+\int_{-\tau}^0\|y(t+\theta)\|_X^2\mathrm d\theta\right)\mathrm dt
    &\leq& \int_0^T \|y(t)\|_X^2\mathrm dt+
    \int_{-\tau}^0\int_{\theta}^{T+\theta}\|y(\sigma)\|_X^2 \mathrm d\sigma \mathrm d\theta\\
     &\leq& (1+\tau)T \max_{t\in[0,T]}\|y(t)\|_X^2+\tau\int_{-\tau}^0
    \|\varphi(\sigma)\|_X^2\mathrm d\sigma.
\end{array}
\end{equation}
   Take $z_1,z_2\in \mathcal{E}_{T,y_0,\varphi}$ arbitrarily. Since $F$ is bounded and $z_1-z_2\in\mathcal{E}_{T,0,0}$, it follows from
    \eqref{yu-11-27-2} and \eqref{yu-11-26-12} that
\begin{equation*}\label{yu-11-26-17}
\max_{t\in[0,T]}\|\mathcal{G}[z_1](t)-\mathcal{G}[z_2](t)\|_X\leq
    \sqrt{(1+\tau)TC(\tau)}\|F\|_{\mathcal{L}(X\times L^2(-\tau,0;X);U)}\max_{t\in[0,T]}\|z_1(t)-z_2(t)\|_X.
\end{equation*}
    This, along with \eqref{yu-11-26-15}, implies that $\mathcal{G}$ is a contraction mapping from
    $\mathcal{E}_{T,y_0,\varphi}$ to itself. Then by Banach fixed-point theorem, there exists a unique $y^1\in \mathcal{E}_{T,y_0,\varphi}$
    so that $y^1(t)=\mathcal{G}[y^1](t)$ for each $t\in[0,T]$, which, combined with the definition of $\mathcal{G}$, indicates that the equation:
\begin{equation*}\label{yu-11-26-18}
    y(t)=S(t)y_0+\int_0^tS(t-s)E\varphi(s-\tau)\mathrm ds
    +\int_0^t\widetilde{S}(t-s)B[F(y(s),
    y(s+\cdot))^\top]\mathrm ds,
    \;\;t\in[0,T]
\end{equation*}
   has a unique solution $y^1$ in $\mathcal{E}_{T,y_0,\varphi}$. By a standard stepwise extension argument, one constructs a unique solution on
$[0,\tau]$ by concatenating solutions on subintervals of length
$T$. Iterating this procedure yields a function
$y\in C([0,+\infty);X)$ solving \eqref{yu-11-26-11}.

The uniqueness follows from uniqueness on each subinterval.
The proof is complete.
    \end{proof}
\begin{remark}\label{remark-11-27-1}
    When $F\in\mathcal{L}(X;U)$, it can be regarded as an element in $\mathcal{L}(X\times L^2(-\tau,0;X);U)$. Hence, by Proposition
    \ref{yu-proposition-11-26-2}, we  conclude that for each $y_0\in X$ and $\varphi\in L^2(-\tau,0;X)$, the equation:
\begin{equation*}
    \begin{cases}
    y_t(t)=Ay(t)+Ey(t-\tau)+BFy(t),&t\in\mathbb{R}^+,\\
    y(0)=y_0,\\
    y(t)=\varphi(t),&t\in(-\tau,0)
\end{cases}
\end{equation*}
    has a unique solution in $C([0,+\infty);X)$.
\end{remark}

\end{document}